\documentclass[
a4paper,11pt]{article}
\usepackage[margin=1in]{geometry}

\usepackage[
colorlinks,
linkcolor=blue,
urlcolor=magenta,
citecolor=red]{hyperref}
\usepackage{amsthm}
\usepackage{amsmath}
\usepackage{amssymb}
\usepackage{xcolor}
\usepackage{upgreek}
\usepackage{mathrsfs}
\usepackage{graphicx}

\usepackage{physics}
\usepackage{cleveref}

\usepackage{mathtools}

\newcommand{\D}{\ensuremath{{\mathscr{D}}}}

\newcommand{\loc}{\ensuremath{\text{loc}}}

\newcommand{\mb}[1]{\ensuremath{\mathbb{#1}}}
\newcommand{\N}{\mb{N}}

\newcommand{\R}{\mb{R}}

\newfont{\bl}{msbm10 scaled \magstep2}

\newcommand{\beq}{\begin{equation}}
\newcommand{\eeq}{\end{equation}}

\newcommand{\notmid}{\mid\kern-0.5em\not\kern0.5em}

\newcommand{\eps}{\varepsilon}

\newcommand{\supp}{\mathop{\mathrm{supp}}}

\newtheorem{thm}{Theorem}[section]
\newtheorem{lem}[thm]{Lemma}
\newtheorem{prop}[thm]{Proposition}

\newtheorem{cor}[thm]{Corollary}

\theoremstyle{definition}
\newtheorem{rem}[thm]{Remark}
\newtheorem{defi}[thm]{Definition}

\newcommand{\vol}{\mathrm{vol}}

\newcommand{\Ric}{\mathrm{Ric}}
\newcommand{\m}{\mathfrak{m}}

\newcommand{\tmcp}{\textsf{TMCP}}

\renewcommand{\labelenumi}{(\roman{enumi})}
\renewcommand\theenumi\labelenumi

\newcommand{\dint}{\mathrm{d}}

\newcommand{\reps}{\rho_\varepsilon}

\newcommand{\Lip}{\mathrm{Lip}}

\usepackage{wasysym}
\DeclareMathOperator*{\esssup}{ess\,sup}
\allowdisplaybreaks

\title{Hawking's singularity theorem for Lipschitz Lorentzian metrics}
\author{Matteo Calisti\thanks{{\tt matteo.calisti@univie.ac.at}, Faculty of Mathematics, University of Vienna, Oskar-Morgenstern-Platz 1, 1090 Wien, 
Austria.}, Melanie Graf\thanks{{\tt melanie.graf@uni-hamburg.de}, Fachbereich Mathematik, Universit\"at Hamburg, Bundesstraße 55, 20146 Hamburg, Germany},
Eduardo Hafemann\thanks{{\tt eduardo.hafemann@uni-hamburg.de},  Fachbereich Mathematik, Universit\"at Hamburg, Bundesstraße 55, 20146 Hamburg, Germany},\\
Michael Kunzinger\thanks{{\tt michael.kunzinger@univie.ac.at}, Faculty of Mathematics, University of Vienna, Oskar-Morgenstern-Platz 1, 1090 Wien, 
Austria.}, Roland Steinbauer\thanks{{\tt roland.steinbauer@univie.ac.at}, Faculty of Mathematics, University of Vienna, Oskar-Morgenstern-Platz 1, 1090 Wien, Austria.}}

\begin{document}
\maketitle

\begin{abstract}
     We prove Hawking's singularity theorem for spacetime metrics of local Lipschitz regularity. The proof rests on (1) new estimates for the Ricci curvature of regularising smooth metrics that are based upon a quite general Friedrichs-type lemma and (2) the replacement of the usual focusing techniques for timelike geodesics---which in the absence of a classical ODE-theory for the initial value problem are no longer available---by a worldvolume estimate based on a segment-type inequality that allows one to control the volume of the set of points in a spacelike surface that possess long maximisers.   
 \end{abstract}
\emph{MSC2020: }83C75; 53C50, 46T30\\
\emph{Keywords: }Singularity theorems, low regularity, regularisation, causality theory

\tableofcontents

\section{Introduction}\label{sec:intro}

The singularity theorems form an important body of results in Lorentzian differential geometry that firmly establish the occurrence of spacetime singularities as a generic feature of General Relativity (GR). More precisely they assert causal geodesic incompleteness under different sets of physically reasonable conditions like those associated with complete gravitational collapse \cite{Pen:65} or an expanding universe \cite{Haw:67}. Naturally these results, which appeared roughly in the second half of the 1960-ies, were formulated for smooth spacetimes. However, already Hawking and Ellis in their classic \cite{HE:73} discussed the issue of regularity: A lack of low-regularity versions of the theorems, that is for spacetime metrics $g$ below the $C^2$-class, would undermine their physical significance. Indeed, then incompleteness and hence a spacetime singularity in the sense of the standard definition (see e.g.\ \cite[p.\ 10]{Cla:98} and \cite[Sec.\ 8.1] {HE:73} for a discussion) could be avoided by a drop in regularity of $g$: If, for example, $g$ were in $C^{1,1}$, then via the field equations there would be a finite jump in the matter variables, which hardly could be termed `singular' on physical grounds. Therefore \cite[Sec.\ 8.4]{HE:73} contains an in-depth discussion of these issues in which the authors argue for a $C^{1,1}$-version of Hawking's theorem \cite[Sec.\ 8.2, Thm.\ 4]{HE:73} and express their expectation that also all the other classical theorems would extend to this regularity. Accepting this for the moment one may observe that given these results, the alternative to incompleteness would now be locally unbounded curvature. This, however, still might be manageable as long as one can make sense of the field equations, which, in particular, is possible if $g$ is locally Lipschitz continuous, i.e., $g\in C^{0,1}$, as again highlighted in \cite[Sec.\ 8.4]{HE:73}. In particular, this includes very prominent classes of solutions like impulsive gravitational waves (e.g.\ \cite{Pen:72a}, \cite[Ch.\ 20]{GP:09}, \cite{PS:22}), thin shells and many matched spacetimes (e.g.\ \cite{Isr:66,KS:23z,MM:24} ) which all exhibit a $\delta$-function like curvature concentrated on a hypersurface. This is precisely the regularity class we deal with in this work, where we extend the validity of the Hawking singularity theorem to spacetime metrics $g\in C^{0,1}$.

While the singularity theorems' regularity issues have been duly addressed throughout the decades, see e.g.\ \cite[Sec.\ 6.2]{Sen:98}, real progress only has been made rather recently. 
By extending Lorentzian causality theory (mainly) to continuous spacetime metrics \cite{CG:12,FS:12,Min:15,Sae:16,KSSV:14} and to even more general settings \cite{Min:19a,KS:18} on the one hand, and by sharpening the analytic tool of approximation via convolution on the other, the classical theorems could be proven for  $C^{1,1}$-metrics \cite{KSSV:15,KSV:15,GGKS:18} and later also for $g\in C^1$ \cite{Gra:20,SS:21,KOSS:22}. Here we take this endeavour one decisive step further, namely to $g\in C^{0,1}$. 

Let us briefly address the added difficulties we have to deal with in this regularity. While for $C^1$-metrics the initial value problem for the geodesic equation is still solvable, albeit not uniquely so, we here face the problem that its right hand side is merely locally bounded. This forces us to work with maximising curves rather than with solutions to the ODE and so the ODE-techniques decisively used in the $C^1$-results \cite[Sec.\ 2]{Gra:20} to approximate the geodesics of the rough metric by geodesics of approximating metrics are no longer at our disposal. This forces us to significantly modify the analytical core of the proof, which provides the focusing of geodesics under curvature bounds. While we still rely on a regularisation scheme that allows us to exploit the distributional strong energy condition, we will then apply the newly developed worldvolume estimates of \cite{GKOS:22} to the smooth regularising metrics. This will allow us to control (the volume of) the set of points on a Cauchy surface that possess long maximisers, but on the other hand forces us to sharpen our estimates on regularised curvature. Indeed, exploiting Friedrichs lemma-type arguments we provide new $L^p$-estimates on the Ricci curvature for Lipschitz metrics and their regularisation, which constitute the main analytical advance of our work. Moreover, we also have to deal with the `initial condition' of the theorem, i.e., a bound on the mean curvature of a spacelike hypersurface. We extend the standard condition to Lipschitz regularity by a smearing out of the hypersurface. 
\medskip

Finally, let us address the potential extensions of our methods to the other singularity theorems within Lipschitz regularity. In the Penrose case, where one is concerned with null geodesics emanating from a trapped surface, one would have to replace the usual focusing estimates by a null version of the segment-type inequality. While this seems feasible in principle, it cannot be expected to be a simple variation of the timelike analysis of \cite{GKOS:22}. Moreover, the very notion of a trapped surface would have to be adapted in a way similar to the  smearing out of the (Cauchy-) hypersurface in the present work. Additionally, the regularisation approach to leverage the distributional null energy condition to gain estimates on the respective approximations is technically more demanding, cf.\ the $C^1$-case in \cite{Gra:20,KOSS:22}---all issues we reserve for future research. Finally, the Hawking-Penrose theorem already classically relies on a more involved focusing analysis providing the existence of conjugate points along a single causal geodesic solely under an energy and a genericity condition which, from our present perspective, seems particularly difficult to extend to the Lipschitz case.

\medskip

We organise our work as follows. After stating our conventions, we collect some necessary prerequisites in Section \ref{sec:preliminaries}. This first concerns causality theory for (Lipschitz)-continuous metrics, where we also detail how for $g\in C^{0,1}$ we may exploit results from the more general settings of closed cone structures of \cite{Min:19b} and Lorentzian length spaces \cite{KS:18}. Then we set up the regularisation scheme to be used throughout and, finally, 
introduce the worldvolume estimates of \cite{GKOS:22}, which are inspired by the segment inequality of \cite{CC:96}. In Section \ref{sec:regularisation} we provide regularisation results for distributional curvature. In particular, we establish new $L^p$-convergence ($1\leq p<\infty$) as well as $L^\infty$-boundedness results for the Ricci curvature of a sequence of metrics approximating $g\in C^{0,1}$ (Proposition \ref{Pr:Ricci Lp-convergence}), which are based on the quite general Friedrichs Lemma \ref{order1Fridrichs}. In Section \ref{sec:meancurvature} we will be concerned with the initial condition of the Hawking theorem. We develop a sensible notion of mean curvature of a hypersurface in case $g\in C^{0,1}$ and establish its compatibility with the smooth setting. Finally, in Section \ref{sec:main} we formulate and prove our main result, a Hawking singularity theorem for locally Lipschitz continuous spacetimes. As usual we give two formulations, one in the globally hyperbolic case and one supposing just a (converging) compact spacelike hypersurface. Finally, we also comment on the interrelation of our and related $C^1$-results to several synthetic versions of the singularity theorems that have appeared recently. 


\section{Preliminaries}\label{sec:preliminaries}
In this section we will be concerned with three topics that underlie our main arguments and to which we will devote one subsection each, causality theory, regularisation of distributional tensor fields, and volume comparison.

Our standard references are \cite{ONe:83} for smooth Lorentzian geometry, the seminal \cite{Min:19b} for causality theory, and \cite{Sae:16} for the low-regularity (continuous) setting. We will generally be concerned with spacetimes $(M,g)$ where $M$ is a smooth, Hausdorff, second countable, and connected manifold of dimension $n\geq 2$ and $g$ is a continuous  Lorentzian metric. Whenever $g$ has some added regularity we will state this explicitly. In particular, we will be interested in locally Lipschitz continuous metrics, for which we will write $g\in C^{0,1}$. Anyway, we will always assume that a time orientation is fixed by a smooth timelike vector field and call such a Lorentzian manifold a $C^0$- or $C^{0,1}$-spacetime. We will also equip $M$ with a complete background Riemannian metric $h$ and denote its length metric by $d^h$. Estimates on smooth vector and tensor fields $X\in{\mathfrak X}(M)$ and $T\in{\mathcal T}^r_s(M)$ will always be done w.r.t. the norms induced by $h$. Spaces of test functions will be denoted by $\mathscr{D}$ and distributions by $\mathscr{D}'$, and in particular we will write  $\mathscr{D}'{\mathcal T}^r_s(M)$ for distributional tensor fields.  We refer to 
\cite[Sec.\ 2]{KOSS:22} for a concise overview of distribution theory on manifolds as employed in this work. With a view to the distributional Ricci bounds to be imposed below we recall, in particular, that a scalar distribution $u\in\D'(M)$ is nonnegative, $u\geq 0$, if $u(\omega)\equiv\langle u,\omega\rangle\geq 0$ for all nonnegative test densities $\omega$. Any nonnegative distribution is a measure and hence a distribution of order $0$.

 Finally, we will write $K\Subset M$ if $K$ is a compact subset of $M$, and the regularisation parameter $\eps$ will generally be taken from $(0,1]$.

\subsection{Causality theory}

We base causality notions on locally Lipschitz curves, that is we call a curve $\gamma: I\to M$ defined on some arbitrary non-trivial interval  $I$ \emph{timelike, null, causal, future} or \emph{past directed} if it is locally Lipschitz and $\dot\gamma$ has the respective properties almost everywhere. 
Then $p\ll q$ (resp.\ $p\leq q$) means that there exists a future directed timelike (resp.\ causal) curve from $p$ to $q$, $I^+(A):=\{q\in M:\, p\ll q\ \mathrm{for\,some}\,p\in A\}$ and $J^+(A):=\{q\in M:\, p\leq q\ \mathrm{for\,some}\,p\in A\}$. 

We call $(M,g)$ \emph{globally hyperbolic} if it is non-totally imprisoning (no future or past inextendible causal curve is contained in a compact set) and for all pairs of  points $p,q\in M$ the causal diamonds $J(p, q):= J^+ (p)\cap J^-(q)$ are compact. Note that this implies that $M$ is strongly causal by \cite[Prop.\ 5.6]{Sae:16}.
A subset $\Sigma\subseteq M$ is called a \emph{Cauchy hypersurface} if it is met
exactly once by every inextendible causal curve. It is always a closed acausal topological hypersurface, and $(M,g)$ is globally hyperbolic if and only if it possesses a Cauchy hypersurface $\Sigma$, in which case $M$ is homeomorphic to $\R\times\Sigma$, \cite[Sec.\ 5]{Sae:16}.
We define the \emph{Cauchy development} of some set $S$ as
\begin{equation}
    D(S):=D^+(S)\cup D^-(S),
\end{equation}
where $D^\pm(S)$ are the sets of points $p\in M$ such that every future/past-directed future/past inextendible causal curve through $p$ meets $S$, cf. \cite{CG:12}. The interior of the Cauchy development $D(S)^\circ$ of any acausal set $S$ is globally hyperbolic \cite[Cor.\ 5.8]{Sae:16}. Moreover, the Avez-Seifert theorem extends to continuous globally hyperbolic spacetimes, i.e., in such $(M,g)$ there is a globally maximising causal curve between any pair of causally related points (cf. \cite[Prop.\ 6.4]{Sae:16}). However, the relation between maximisers and geodesics becomes more subtle here.

\begin{rem}[Geodesics and maximisers]\label{rem:geomax}\hspace{0px}
\begin{enumerate}
    \item For $g\in C^{0,1}$ the right hand side of the geodesic equation is merely locally bounded and hence we are outside classical ODE-theory. However, the initial value problem has solutions in the sense of Filippov \cite{Fil:88} which are $C^1$-curves \cite{Ste:14}, satisfying a differential inclusion relation for the essential convex hull of the locally bounded right hand side. But solutions of the geodesic equations may fail to be local maximisers already for $g$ in the Hölder class $C^{1,\alpha}$ for $\alpha<1$ \cite{HW:51,SS:18}. 
    \item Conversely, for $g\in C^{0,1}$ any maximiser, when parametrized by $g$-arclength is a Filippov-geodesic of regularity $C^{1,1}$, see \cite[Thm.\ 1.1]{LLS:21} and the discussion following it. Also, such maximisers are either timelike or null throughout \cite{GL:18,LLS:21}.    
\end{enumerate}
\end{rem}

For $p,q\in M$ the future \emph{time separation} is defined by 
\begin{equation}
    \tau(p,q):=\sup(\left\{ L(\gamma):\gamma\;\text{is a future directed causal curve from }p\text{ to }q\right\} \cup\{0\}),\label{eq:point time sep}
\end{equation}
where $L(\gamma)$ denotes the Lorentzian arc-length of $\gamma: I\to M$, i.e., $L(\gamma):=\int_I\sqrt{|g(\dot{\gamma}(t),\dot{\gamma}(t))|}dt$.  Moreover, one defines the future time separation to a subset $S$ by
\begin{equation}
    \tau_{S}(p):=\sup_{q\in S}\tau(q,p).\label{eq:subset time sep}
\end{equation}

While it is known that for continuous metrics basic features of causality theory (such as the push up principle and the openness of $I^+$) break down \cite{CG:12,GKSS:20}, this is not the case for the class of causally plain metrics, to which any $g\in C^{0,1}$ belongs \cite[Thm.\ 1.20]{CG:12}. Moreover, although we will be concerned with the (low regularity) spacetime setting, we will freely use results from causality theory derived in more general settings, in particular the closed cone structures of \cite{Min:19a} and the Lorentzian (pre-)length spaces of \cite{KS:18}. Here we briefly recall how continuous spacetime metrics $g$ fall into these settings.

\begin{rem}[Cone structures]\label{rem:cs}
    Causality in cone structures, which generalise Lorentzian causality theory to an order theoretic setting, were introduced by Ettore Minguzzi in the seminal paper \cite{Min:19a}. A \emph{cone structure} $(M,C)$ is a multivalued map $M\ni p\mapsto C_p$, where $C_p\subseteq T_pM\setminus 0$ is a closed sharp convex non-empty cone, \cite[Definition 2.2]{Min:19a}. It is called \emph{closed} \cite[Def.\ 2.3]{Min:19a}, if it is a closed
    subbundle of the slit tangent bundle. 
    
    Let $g$ be a locally Lipschitz metric, and $C_p:=\{v\in T_pM\,\backslash\,\{0\}:g(v,v)\leq0, v\mbox{ future directed}\}$ as in \cite[Ex.\ 2.1]{Min:19a}, then by \cite[Prop.\ 2.4]{Min:19a} the map $p\mapsto C_p$ is locally Lipschitz and by \cite[Prop.\ 2.3]{Min:19a} $(M,C)$ is a closed cone structure. In addition it is proper, i.e., closed with all $C_q$ having nonempty interior.
\end{rem}

\begin{rem}[Lorentzian length spaces]\label{rem:lls}
    Lorentzian length spaces were introduced by Kunzinger and Sämann in \cite{KS:18} as an analogue of metric length spaces. Let $(X,d)$ be a metric space endowed with a preorder $\leq$ as well as a transitive relation $\ll$ contained in $\leq$, called the timelike and causal relation. If in addition we have a lower semicontinuous map $\rho \colon X\times X \to [0, \infty]$ that satisfies the reverse triangle inequality and $\rho(x,y)>0 \Leftrightarrow x\ll y$, then $(X,d,\ll,\leq,\rho)$ is called a \emph{Lorentzian pre-length space\/} with  \emph{time separation function\/} $\rho$.
    
    The length of a future-directed causal $\gamma \colon [a,b]\rightarrow X$ (i.e. $t_1<t_2$ implies $\gamma(t_1)\leq\gamma(t_2)$) is defined as $L_\rho(\gamma):=\inf\{\sum_{i=0}^{N-1} \rho(\gamma(t_i),\gamma(t_{i+1})): a=t_0<t_1<\ldots<t_N=b,\ N\in\N\}$, 
    and $(X,\rho)$ is called a \emph{Lorentzian length space} if, in addition to some technical assumptions (cf.\ \cite[Def. 3.22]{KS:18}) $\rho$ is intrinsic, i.e., $\rho(p,q)= \sup\{L_\rho(\gamma):\gamma$ future-directed causal from $p$ to $q\}$.

    By \cite[Theorem 5.12]{KS:18} every continuous, strongly causal and causally
   plain spacetime (and hence every strongly causal Lipschitz spacetime, cf.\ \cite[Thm.\ 1.20]{CG:12})
   is a (strongly localisable) Lorentzian length space (with $\leq,\ll$ as usual, $\rho=\tau$, and $d=d^h$).
   
\end{rem}
\medskip

\subsection{Regularisation of distributional tensor fields
}\label{ssec:mfreg}

A key technical tool throughout this work is regularisation of distributional tensor fields, which we introduce next. Let $(U_\alpha,\psi_\alpha)$ ($\alpha\in \N$) be a countable and locally finite family of relatively compact chart neighbourhoods covering $M$ and pick a subordinate partition of unity $(\xi_\alpha)_\alpha$ with $\mathrm{supp}(\xi_\alpha)\subseteq U_\alpha$ for all $\alpha$. Also, choose a family of cut-off functions $\chi_\alpha\in\mathscr{D}(U_\alpha)$ with $\chi_\alpha\equiv 1$ on a
neighbourhood of $\mathrm{supp}(\xi_\alpha)$. Let $\rho\in \D(B_1(0))$ be a non-negative test function with unit integral and set,
for $\eps\in (0,1]$, $\rho_{\eps}(x):=\eps^{-n}\rho\left (\frac{x}{\eps}\right)$. Then denoting by $f_*$ (resp.\ $f^*$) push-forward (resp.\ pull-back) of distributions under a diffeomorphism $f$, for any tensor distribution $\mathcal{T} \in \mathscr{D}'\mathcal{T}^r_s(M)$ we set
\begin{equation}\label{eq:M-convolution}
\mathcal{T}\star_M \rho_\eps(x):= \sum\limits_\alpha\chi_\alpha(x)\,(\psi_\alpha)^*\Big[\big((\psi_{\alpha})_* (\xi_\alpha\cdot \mathcal{T})\big)*\rho_\eps\Big](x).
\end{equation}
Here, $(\psi_{\alpha})_* (\xi_\alpha\cdot \mathcal{T})$ is a compactly supported distributional tensor field on $\R^n$, and convolution with $\rho_\eps$ is understood component-wise,  yielding a smooth field on $\R^n$. The cut-off functions $\chi_\alpha$ ensure that $(\eps,x) \mapsto \mathcal{T}\star_M \rho_\eps(x)$
is a smooth map on $(0,1] \times M$. For any compact set $K\Subset M$ there is an $\eps_K$ such that for all $\eps<\eps_K$ and all $x\in K$, equation \eqref{eq:M-convolution} reduces to a finite sum with all $\chi_\alpha\equiv 1$ (which can therefore be omitted from the formula), namely when $\eps_K$ is less than the distance between the support of $\xi_\alpha\circ\psi_\alpha^{-1}$ and the boundary of $\psi_\alpha(U_\alpha)$
for all $\alpha$ with $U_\alpha\cap K\neq \emptyset$.

Since the mollifier $\rho$ above is assumed to be nonnegative, it follows that for any nonnegative scalar distribution $u\in \D'(M)$ we also have $u\star_M \rho_\eps \ge 0$ for any $\eps\in (0,1]$.

We next collect basic convergence properties of regularisations of Lipschitz-continuous Lo\-rent\-zian metrics and their Ricci-curvature. To this end we introduce the following notation that we shall use throughout:  $\ast$ will exclusively denote convolutions on $\R^n$, while 
$\star_M$ stands for the manifold convolution \eqref{eq:M-convolution}. We will write 

\begin{equation}\label{eq:geps_def}
g_\eps:=g\star_M \rho_\eps,
\end{equation}

but to avoid confusion we will otherwise not use the subscript $\eps$ to denote quantities derived from $g_\eps$. Instead we will be more explicit and write e.g.\ $\Ric[g_\eps]$ for the Ricci curvature derived from the metric $g_\eps$ but e.g.\ $\Ric[g]\star_M\rho_\eps$ for the manifold convolution of the Ricci curvature of $g$. Finally, for two Lorentzian metrics $g_1,g_2$ on $M$ we write $g_1\prec g_2$ and say that $g_1$ has \emph{strictly narrower light cones} than $g_2$, if for all non-vanishing vectors $X$
\begin{align}\label{Eq:prec}
g_1(X,X)\leq0\quad \text{implies}\quad g_2(X,X)<0.
\end{align}

\begin{lem}[Convergence of approximating metrics]\label{Le:approximating metrics}
Let $g\in C^{0,1}(M)$ be a Lorentzian metric and 
let $g_\eps$ be as in \eqref{eq:geps_def}. Then there are smooth Lorentzian metrics $\hat g_\eps$ and $\check g_\eps$ on $M$ with the following properties
\begin{itemize}
    \item[(i)] $\check{g}_\varepsilon\prec g \prec\hat{g}_\varepsilon$.
    \item[(ii)] $\check g_\eps$, $\hat g_\eps\to g$, and $(\check g_\eps)^{-1}$, $(\hat g_\eps)^{-1}\to g^{-1}$ locally uniformly and in $W^{1,p}_{\mathrm{loc}}(M)$ for all $1\leq p<\infty$.
    \item[(iii)] $\check g_\eps-g_\eps \to 0$, $\hat g_\eps-g_\eps\to 0$, and
    $(\check g_\eps)^{-1}-(g_\eps)^{-1}\to 0$,  $(\hat g_\eps)^{-1}-(g_\eps)^{-1}\to 0$, all in $C^\infty_{\mathrm{loc}}(M)$.
    In particular, $\Ric[\check g_\eps] - \Ric[g_\eps] \to 0$ in $C^\infty_{\mathrm{loc}}(M)$.
    \item[(iv)] For any compact subset $K\Subset M$ there exists a sequence $\varepsilon_j\searrow0$ such that $\check g_{\eps_j}\prec\check g_{\eps_{j+1}}$ and $\hat g_{\eps_{j+1}}\prec \hat g_{\eps_j}$ for all $j\in\N$.
\end{itemize}
\end{lem}

\begin{proof}
Statement (i) is \cite[Lemma 4.2(i)]{Gra:20}. The claims about the metrics in (ii) and (iii) follow from the proof of \cite[Lemma 4.2(iii)]{Gra:20}, only observing that for any Lipschitz function $f$ on an open subset of $\R^n$ and a standard mollifier $\rho_\eps$ we have that $f * \rho_\eps\to f$ locally uniformly and in any $W^{1,p}_{\mathrm{loc}}$, $1\leq p<\infty$. The properties of the inverses in (ii) and (iii) follow from (i), (ii), together with the cofactor formula of matrix inversion and the fact that $\det g$ is bounded away from $0$ on any compact set. Finally, (iv) is \cite[Prop. 2.3(i)]{KSV:15}.
\end{proof}

\subsection{Volume comparison}\label{sec:vc}
To formulate the proofs of our main results and to understand the singularity-theorem adjacent result of \cite[Thm.\ 4.1]{GKOS:22}
which they rely on,
we need to introduce some corresponding notions. Since we will apply these results to smooth regularisations of the rough metric we assume $g$ to be smooth throughout this subsection. 

We consider a smooth spacelike hypersurface $\Sigma$ with unit future normal $\vec{n}$ and corresponding mean curvature bounded by $H\leq\beta<0$. Moreover, we assume a lower bound on the Ricci curvature of $(M,g)$ in timelike unit directions, i.e.,  $\Ric[g](X,X)\geq \kappa n$, where $\kappa$ is any negative number (with potentially large modulus). We denote by $\exp_\Sigma^+:\mathcal{I}^+\to M$ the future normal exponential map to $\Sigma$. Here $\mathcal{I}^+=\{(t,x)\in[0,s^+(x))\times\Sigma\}$ where $[0,s^+(x))$ is the maximal domain of definition of the unique geodesic $\gamma_x$ with initial data $\gamma_x(0)=x$ and $\dot\gamma_x(0)=\vec{n}_x$.
Further we write $c^+_\Sigma:\Sigma\to(0,\infty]$ for the future cut function of $\Sigma$, i.e. $c^+_\Sigma(x)=\sup\{t\in [0,s^+(x)): \tau_\Sigma(\exp_x(t\vec{n}_x))=t\}$. Then for $T>0$ and $B\subseteq \Sigma$ we define the future evolution 
\begin{equation}\label{eq:Omega_T_+}
    \Omega^+_T(B):=\{\exp_x(t\vec{n}_x): x\in B,\, t\in [0,T]\cap[0,c^+_\Sigma(x))\}
\end{equation}
and for $\eta>0$ (considered small as compared to $T$) the set of $(T+\eta)$-regular points $\mathrm{Reg}^+_\eta(T)$ of $x\in\Sigma$ such that $s^+(x)>T+\eta$ and $\gamma_x$ is maximising on $[0,T+\eta]$, i.e.,
\begin{equation}
   \mathrm{Reg}^+_\eta(T)=(c^+_\Sigma)^{-1}([T+\eta,\infty]).
\end{equation}
Observe that for any  $B\subseteq\mathrm{Reg}^+_\eta(T)$ we have that $\Omega^+_T(B)=\exp_\Sigma^+([0,T]\times B)$. Now volume comparison techniques lead to the following segment-type inequality (\cite[Prop.\ 3.10]{GKOS:22}) 
for continuous $f\geq 0$ and $B\subseteq \mathrm{Reg}_\eta^+(T)$ with finite volume (i.e., $0<\sigma(B)<\infty$), where $\sigma$ denotes the volume measure on $\Sigma$ with respect to the Riemannian metric induced by $g$):
\begin{equation}\label{eq:sie}
    \inf_{x\in B}\int_0^{\min(T,s^+(x))}f(\exp^+_\Sigma(t,x))\,\dint t
    \ \leq\ \frac{1}{C^{A-}\sigma(B)}\int_{\Omega^+_T(B)} f\,\mathrm{dvol}_g.
\end{equation}
Here the so-called backwards area comparison constant is explicitly given by
$C^{A-}(n,\kappa,\eta,T) = \big(\mathrm{sinh}(\eta\sqrt{|\kappa|})/\mathrm{sinh}((T+\eta)\sqrt{|\kappa|})\big)^{n-1}$. Applying inequality \eqref{eq:sie} to the negative part of the Ricci curvature evaluated on the future unit congruence of $\Sigma$, i.e., $\Ric[g](U,U)_-$ with 
\begin{equation}\label{eq:Up}
U(p)=\left.\frac{\dint}{\dint s}\right|_{s=0} \exp_x((t+s)\,\vec{n}_x)
\end{equation}
for $p=\exp^+_\Sigma(t,x)$, one has the following result, which is a slightly modified version of \cite[Thm.\ 4.1]{GKOS:22}.

\begin{thm} \label{thm:sat}
Let $(M,g)$ be a smooth globally hyperbolic spacetime with smooth spacelike Cauchy surface $\Sigma$ with $H\leq\beta$ and $\Ric[g](X,X)\geq n\kappa$ for all timelike unit vectors $X$, where $\kappa,\beta<0$ and
   $\beta\geq-(n-1)\sqrt{\abs{\kappa}}$. If  for $B\subseteq\Sigma$ with $0<\sigma(B)<\infty$ and some $T,\eta>0$ and $\rho \in \mathbb{R}$ we have
   \begin{equation}\label{eq:thm2.4}
       \frac{1}{\sigma(B)}\int_{\Omega^+_T(B)}\Big(\Ric[g](U,U)-(n-1)\rho \Big)_- \, \mathrm{dvol}_g\  < \ C^{A-}(n,\kappa,\eta,T) K(\beta,T,\rho),
   \end{equation}
   where
   \begin{equation}\label{eq:K-def}
         K(\beta,T,\rho):=  \begin{cases}
         |\beta|-\frac{n-1}{T} &\text{if} \;\rho=0,\\
         |\beta|-(n-1)\sqrt{|\rho|}\coth(\sqrt{|\rho|}T) &\text{if} \;\rho<0,\\
         |\beta| - (n-1)\sqrt{\rho}\cot(\sqrt{\rho}T) &\text{if} \; \rho>0\  \text{ and }\  \sqrt{\rho}T\le \frac{\pi}{2},
       \end{cases}
   \end{equation}
   then $B\not\subseteq\mathrm{Reg}_\eta^+(T)$.
\end{thm}

\begin{proof}
Proceeding as in the proof of the original \cite[Thm.\ 4.1]{GKOS:22}, we assume to the contrary that $B  \subseteq\mathrm{Reg}_\eta^+(T)$. We set $f(p):=\Big(\Ric(U_p,U_p)-(n-1)\rho \Big)_- $, then the segment-type inequality \eqref{eq:sie} and the assumed estimate \eqref{eq:thm2.4} imply that
there exists $x \in B$ such that
\begin{align}
    \int_0^{\min(T,s^+(x))}f(\exp^+_\Sigma(t,x))\,\dint t =\int_0^T f(\exp^+_\Sigma(t,x))\,\dint t < K(\beta, T,\rho ).
\end{align}
Let $\gamma := \exp_{\Sigma}^+(.,x):[0,T]\to M$ denote the unit speed future normal geodesic to $\Sigma$ starting in $x$, which maximizes the $\Sigma$-time separation up to $p:=\gamma(T) \in I^+(\Sigma)$ since $x\in B\subseteq \mathrm{Reg}_\eta^+(T)$. The standard second variation of arc-length computations along $\gamma$ yields 
\begin{align}
   \nonumber 0 &\geq |\beta| + \int_0^T -(n-1)\dot{h}(t)^2+ h(t)^2 \Ric(\dot{\gamma}(t),\dot{\gamma}(t)) dt =\\\nonumber
   &= |\beta| + \int_0^T -(n-1)\dot{h}(t)^2+ h(t)^2 (n-1)\rho \,\dint t + \int_0^T h(t)^2 \big( \Ric(\dot{\gamma}(t),\dot{\gamma}(t))-(n-1)\rho\big) \dint t\\\nonumber
   & \geq |\beta| + \int_0^T -(n-1)\dot{h}(t)^2+ h(t)^2 (n-1)\rho \,\dint t - \int_0^T h(t)^2 \big( \Ric(\dot{\gamma}(t),\dot{\gamma}(t))-(n-1)\rho\big)_- \dint t \\
   &> |\beta| + \int_0^T -(n-1)\dot{h}(t)^2+ h(t)^2 (n-1)\rho \,\dint t - K(\beta,T,\rho) \label{eq:17}
\end{align}
for any smooth $h:[0,T]\to \R$ with $h(0)=1$, $h(T)=0$ and $|h|\leq 1$. Choosing $h(t)=1-\frac{t}{T}$ in case $\rho=0$, $h(t)=\frac{1}{\sinh(\sqrt{|\rho|}T)}\,\sinh(\sqrt{|\rho|}(T-t))$ in case $\rho <0$, and
$h(t)=\frac{1}{\sin(\sqrt{\rho}T)}\,\sin(\sqrt{\rho}(T-t))$ in case $\rho>0$ (the condition $\sqrt{\rho}T\le \frac{\pi}{2}$ in this case is required to assure $|h(t)|\le 1$ on $[0,T]$), \eqref{eq:17} evaluates to zero, giving the usual contradiction. 
\end{proof}

As indicated above, we will show estimate \eqref{eq:thm2.4} for $\Ric[\check g_\eps]$ and so the task we take up in the next section is to develop the corresponding estimates on the Ricci curvature of the approximations.

\section{Curvature estimates}\label{sec:regularisation}

The main strategy in our regularisation approach is to employ the energy condition for the Lipschitz metric $g$, i.e., the condition that $\Ric[g]$ is a non-negative distribution to derive local curvature bounds on the regularised metrics $g_\eps$. However, since convergence of $\Ric[g_\eps]$ (and so by \Cref{Le:approximating metrics}(iii) $\Ric[\check g_\eps]$, which is the more relevant quantity in our approach) to $\Ric[g]$ is merely distributional, we will exploit the non-negativity of $\Ric[g]\star_M\rho_\eps$ instead. To make use of this property we have to control the difference between the latter quantity and $\Ric[g_\eps]$. Deriving the required estimates is the main aim of this section. More precisely we are going to establish the following result:

\begin{prop}\label{Pr:Ricci Lp-convergence}
    Let $(M,g)$ be a Lorentzian manifold with a locally Lipschitz metric $g$. Then for $g_\eps(=g\star_M \rho_\eps)$ we have for any compact $K\Subset M$
    \begin{itemize}
    \item[(i)] $\Vert\Ric[g_\varepsilon]-\Ric[g]\star_M \rho_\varepsilon\Vert_{L^p(K)}\rightarrow 0$ for all $1\leq p<\infty$, and
    \item[(ii)]  there exists some $C_K>0$ such that for $\eps$ small enough
    \begin{equation*}
        \smash{\Vert\Ric[g_\varepsilon]-\Ric[g]\star_M \rho_\varepsilon\Vert_{L^\infty(K)}\leq C_K.}
    \end{equation*}
    \end{itemize}
\end{prop}
As a first step towards a proof of Proposition \ref{Pr:Ricci Lp-convergence}, note that explicitly calculating from \eqref{eq:M-convolution} the push-forward under a chart of $g_\eps$, the local expressions of the relevant terms in $\Ric[g_\varepsilon]-\Ric[g]\star_M \rho_\varepsilon$ containing all second order derivatives of $g$ take the form (cf.\ \cite[proof of Lemma 4.6]{Gra:20}) of first order derivatives of
\begin{equation}\label{incrediblequat}
\begin{split}
\big[(\psi_\beta)_*g_\varepsilon\big]^{ij}
\big(
\big[\xi\partial_k((\psi_\beta)_*g)_{lm}\big]
*\rho_\varepsilon\big)-\big(
\big[(\psi_\beta)_*g\big]^{ij}
\xi\partial_k(( & \psi_\beta)_*g)_{lm}
\big)*\rho_\varepsilon \\ 
&=:a_\varepsilon (f*\reps)-(af)*\reps,
\end{split}
\end{equation}
where $\psi_\beta$ is a coordinate chart, $\xi$ a cutoff function,  $\rho_\varepsilon$ a standard mollifier, and $*$ denotes the usual convolution on $\R^n$.
According to \eqref{eq:M-convolution} and the remarks following it, in a neighbourhood of any $K\Subset M$ we can write
\begin{equation}
 (\psi_\beta)_*g_\eps = \sum_\alpha (\psi_\alpha\circ \psi_\beta^{-1})^*(((\psi_\alpha)_*(\xi_\alpha g))*\rho_\eps).   
\end{equation}
Now set $\tilde g_\alpha:=(\psi_\alpha)_*(\xi_\alpha g)$. Then if each $\tilde g_\alpha * \rho_\eps$ converges in $W^{1,p}_{\mathrm{loc}}\cap L^\infty_{\mathrm{loc}}$, so does  $(\psi_\beta)_*g_\eps$. Thus for all further calculations in this section we may without loss of generality consider the case $M=\R^n$ with the single chart $\psi_\alpha=\mathrm{id}$ and $g_\eps=g*\rho_\eps$. Then $a_\eps$ becomes a component of the inverse of the smoothed metric $(g*\rho_\eps)^{ij}$, and $a$ is a component of $g^{ij}$. Also, $f=\xi(\partial_k g)_{lm}$ is a compactly supported $L^\infty$-function. Since $a_\varepsilon\neq a*\rho_\varepsilon$, we reserve the notation $a_\varepsilon$ just for that term and write the convolution explicitly for all the others.
Hence, to prove \Cref{Pr:Ricci Lp-convergence} we have to show that for such  $a$, $a_\eps$ and $f$ we have for all $K\Subset \R^n$
\begin{equation}\label{eq:est}
\begin{split} 
    &a_\eps (f*\rho_\eps) - (af)*\rho_\eps \to 0\quad\ \text{in $W^{1,p}(K)$ for each $p\in [1,\infty)$, as well as}\\[.2\baselineskip]
    &\|a_\eps (f*\rho_\eps) - (af)*\rho_\eps\|_{W^{1,\infty}(K)} \le C_K.
\end{split}  
\end{equation}

 We first collect all relevant properties of $a_\eps$ and $a$ to be used below. 

\begin{lem}\label{Le: a and a_eps}
    Let $a_\eps$ and $a$ be components of $(g*\rho_\eps)^{ij}$ and $g^{ij}$, respectively. Then on any compact $K\Subset\R^n$ we have
    \begin{enumerate}
        \item $a$ is Lipschitz on $K$
        \item $a_\eps$ is smooth and Lipschitz on $K$, uniformly in $\eps$
        \item $a_\eps \to a \ \text{in} \ W^{1,p}(K)$ for $1\leq p<\infty$
        \item $|a_\eps(x)-a(x)|\leq C_K\eps$.
    \end{enumerate}
\end{lem}
\begin{proof}
(i) and (iii) follow from the cofactor formula and standard properties of convolution, cf.\ 
 \Cref{Le:approximating metrics}(ii).

To prove (ii), recall first that for any locally Lipschitz function $h$ we have again by standard properties of the convolution that for all small $\eps$
\begin{align}
 \label{Eq: lip of reg on Rn}\Lip(h*\reps,K) \leq \Lip(h,K'),
\end{align}
where $\Lip(h,K')$ is the Lipschitz constant of $h$ on a suitable compact neighbourhood $K'$  of $K$. Next we explicitly write out $a_\varepsilon$ using the cofactor formula 
\begin{equation}
\begin{split}
    \abs{a_\varepsilon(x)-a_\varepsilon(z)}&=\abs{\frac{1}{\mathrm{det}\,(g*\rho_\varepsilon)(x)}((g*\rho_\varepsilon)^{\mathrm{cof}})_{ij}(x)
    -\frac{1}{\mathrm{det}\,(g*\rho_\varepsilon)(z)}((g*\rho_\varepsilon)^{\mathrm{cof}})_{ij}(z)}  \\
    &\le\abs{\frac{1}{\mathrm{det}\,(g*\rho_\varepsilon)(x)}}\ 
    \bigg\vert ((g*\rho_\varepsilon)^{\mathrm{cof}})_{ij}(x)-((g*\rho_\varepsilon)^{\mathrm{cof}})_{ij}(z)\bigg\vert  \\ 
     &\hspace{3cm}+\big| ((g*\rho_\varepsilon)^{\mathrm{cof}})_{ij}(z)\big|\ 
     \bigg\vert\frac{1}{\det(g*\rho_\varepsilon)(x)}
     -\frac{1}{\det(g*\rho_\varepsilon)({z})}\bigg\vert\\
    &
    {\leq}\bigg[\ C_1\ \abs{\frac{1}{\det(g*\rho_\varepsilon)({x})}}
    + \ C_2\ \abs{\frac{{1} 
    }{\det(g*\rho_\varepsilon)(z)\det(g*\rho_\varepsilon)(x)}} 
    \bigg]\ \abs{z-x}, 
\end{split}
\end{equation}
where the constants $C_1$, $C_2$ depend on the Lipschitz constants and the $L^\infty$-bounds of the components of $g$ on a suitable compact neighbourhood $K'$ 
of $K$. Now since $\det g$ is uniformly bounded away from zero on $K'$ and once more by the uniform bounds on $g$ we obtain the uniform Lipschitz property on $K$.

Finally, in order to prove (iv) we again write out the cofactor formula and, by inserting and subtracting terms, we obtain a sum of terms which are products of (uniformly in $\eps$) bounded functions with only a single factor being a difference of a component of $g$ and its convolution with the mollifier. Since the latter is bounded by a constant times $\eps$ the result follows.
\end{proof}

We now begin to prove \eqref{eq:est}. Note that the zero order estimates follow easily from 
standard properties of the convolution and Lemma \ref{Le: a and a_eps}. In fact, we even have for each $K\Subset \R^n$
\begin{equation}\label{eq:0order}
    \Vert a_\varepsilon (f*\rho_\varepsilon)-(af)*\rho_\varepsilon\Vert_{L^\infty(K)}\leq C_{K'}\ \eps\ \Vert f\Vert_{L^\infty(K')},
\end{equation}
which also implies suitable estimates for any remaining terms in $\Ric[g_\eps]-\Ric\star_M \rho_\eps$ containing at most first derivatives of $g$.

The first order estimates are more delicate and we establish them
in the following statement, whose proof follows a layout similar to that of \cite[Appendix A]{MS:02}.
\begin{lem}[Friedrichs Lemma]\label{order1Fridrichs}
    Let $a_\varepsilon$, $a$ be as in Lemma \ref{Le: a and a_eps} and let $f\in L^\infty(\R^n)$ be compactly supported. Then we have for each $K\Subset \R^n$ and each $1\leq j\leq n$
    \begin{itemize}
        \item [(i)] $\left\Vert\partial_j\big(a_\varepsilon (f*\rho_\varepsilon)-(af)*\rho_\varepsilon\big)\right\Vert_{L^p(K)}\rightarrow0$ for all $p\in[1,\infty)$, and        
        \item[(ii)] there is $C_K>0$ such that $\left\Vert\partial_j\big(a_\varepsilon (f*\rho_\varepsilon)-(af)*\rho_\varepsilon\big)\right\Vert_{L^\infty(K)}\leq C_K$. 
    \end{itemize}
\end{lem}

\begin{proof}
To show (i), we start by writing 
\begin{align}\label{incrediblequatsecond}
    \partial_j\Big(a_\eps (f*\rho_\eps) -(af)*\rho_\eps\Big) =
    \partial_j\Big((a_\varepsilon-a)(f*\rho_\varepsilon)\Big) +\partial_j\Big(a(f*\rho_\varepsilon)-(af)*\rho_\varepsilon\Big).
\end{align}
The idea is to write both terms as an integral operator acting on $f$ and to study the properties of the corresponding kernels. We start with the latter term on the r.h.s.\ of \eqref{incrediblequatsecond} and find
\begin{equation}\label{Eq:kernel integrated}
\begin{split}
   \partial_j\Big(a(f*\rho_\varepsilon)&-(af)*\rho_\varepsilon\Big)(x)
   \\
   &=\int_{\R^n} \frac{\partial}{\partial {x^j}}\Big(\big(a(x)-a(y)\big)\rho_\eps(x-y)\Big) f(y)\,\dint y
   =:\, \int_{\R^n}k_\eps(x,y)f(y)\,\dint y,  
\end{split}
\end{equation}
so the operator takes the form $K_\eps f(x)=\int_{\R^n} k_\eps (x,y)f(y)dy$ with 
\begin{equation}\label{eq:kernel}
    k_\eps(x,y)=\partial_{x^j}\Big(\big(a(x)-a(y)\big)\rho_\eps(x-y)\Big).
\end{equation}
Since for fixed $\eps$ the kernel $k_\eps$ is essentially bounded, it gives rise to a bounded operator $K_\eps:L^p(\R^n)\to L^p(K)$ for all $1\leq p<\infty$.
Moreover, the support of $k_\eps$ is contained in an $\eps$-neighbourhood of $\supp(a)\times \supp(a)$, and $k_\eps(x,y)=0$ for $|x-y|>\eps$.
Finally, to establish properties of $K_\eps$ that are uniform in $\eps$  we observe that the kernels $k_\eps$ satisfy:
    There is $C>0$ such that 
    \begin{equation}\begin{aligned}\label{eq:diff-kernel}
        &\int_{K}\vert k_\varepsilon(x,y) \vert\,\dint x\leq C\quad
        \text{for all $y\in \R^n$ and all $\varepsilon>0$,}\\
        &\int_{\R^n}\abs{k_\varepsilon(x,y)}\,\dint y\leq C\quad\text{for all $x\in K$ and all $\varepsilon>0$.}
    \end{aligned}
    \end{equation}
    Indeed writing the kernel as
    \begin{align}\label{eq:kernel-deriv}
        k_\eps(x,y)=\frac{\partial a(x)}{\partial x^j}\,\rho_\eps(x-y)
        +\big(a(x)-a(y)\big)\,\frac{\partial \rho_\eps(x-y)}{\partial x^j}
    \end{align}
    we obtain the estimate 
    \begin{align}\label{eq:kernel-dx}
        \int_{K}|k_\eps(x,y)|\,\dint x\leq 
        \Lip(a,K)+\eps\, \Lip(a,K')\, \frac{1}{\eps} \int \left|\frac{\partial \rho(z)}{\partial z^j}\right|\,\dint z,
    \end{align}
    were $K'$ is a compact neighbourhood of $K$. So we obtain \eqref{eq:diff-kernel} with $C=(1+\int |\nabla \rho|)\Lip(a,K')$ and the same reasoning applies to the integral with respect to the $y$-variable.
    
    From 
    \eqref{eq:diff-kernel} we immediately obtain uniform $L^1$-boundedness of $K_\eps$, and indeed also for $L^p$ ($1<p<\infty$). In fact, taking advantage of both estimates in \eqref{eq:diff-kernel} and Hölder's inequality we have for $1/p+1/q=1$
    \begin{align}\nonumber
        \Vert K_\varepsilon f\Vert_{L^p(K)}^p&=\int_{K}\abs{\int_{\R^n}k_\varepsilon(x,y)f(y)\,\dint y}^p\dint x\leq\int_{K}\bigg(\int_{\R^n}\abs{k_\varepsilon(x,y)^{\frac{1}{p}+\frac{1}{q}}f(y)}\,\dint y\bigg)^p\dint x\\
        &\label{Eq:K L1}\leq \int_{K}   \bigg[\bigg(\int_{\R^n}\abs{k_\varepsilon(x,y)}\abs{f(y)}^p\,\dint y\bigg)^\frac{1}{p}\bigg(\int_{\R^n}\abs{k_\varepsilon(x,y)}\,\dint y\bigg)^\frac{1}{q}\bigg]^p\dint x\\
        \nonumber&\leq C^\frac{p}{q}\int_{\R^n} \abs{f(y)}^p\bigg(\int_{K}\abs{k_\varepsilon(x,y)}\,\dint x\bigg)\,\dint y
         \leq C^p\ \Vert f\Vert^p_{L^p(\R^n)}.
    \end{align}   
Using \eqref{eq:diff-kernel} we will now show that $K_\eps f\to 0$ in $L^p(K)$ for each $f\in L^p(\R^n)$. In fact we would only need to consider $f\in L^\infty(\R^n)$ with compact support. Being a uniformly bounded family of operators it suffices to establish that $\Vert K_\eps f\Vert_{L^p(K)} \to 0$ for any test function $f$ in the dense subspace $C^\infty_c(\R^n)$ of $L^p(\R^n)$, see \Cref{functionalanalysislemma}, below. While this follows from \eqref{incrediblequatsecond} by mostly standard tricks for smoothing by convolution, we include the detailed estimates here for convenience.

So suppose that $f$ is a compactly supported smooth function,  
then we have
 \begin{align}\nonumber
    \int_{K}\vert K_\eps f(x)\vert^p\,\dint x
       &=\int_{K}\bigg\vert\int_{\R^n}\bigg[\frac{\partial a(x)}{\partial x^j}\reps(x-y)f(y)+\big(a(x)-a(y)\big)\frac{\partial\reps(x-y)}{\partial x^j}f(y)\bigg]\,\dint y\bigg\vert^p\dint x\\[.5\baselineskip]\nonumber
    &
    =\int_{K}\bigg\vert\int_{\R^n}\bigg[\frac{\partial a(x)}{\partial x^j}\reps(x-y)\big(f(y){-f(x)}\big)\,{+}\,\frac{\partial a(x)}{\partial x^j}\reps(x-y){f(x)}\\\nonumber
    &\hspace{5.5cm}-\big(a(x)-a(y)\big)\frac{\partial\reps(x-y)}{\partial y^j}f(y)\bigg]\,\dint y\bigg\vert^p\dint x\\
    &
    =\int_{K}\bigg\vert\int_{\R^n}\frac{\partial a(x)}{\partial x^j}\reps(x-y)\big(f(y)-f(x)\big)\,\dint y\\\nonumber
    &\hspace{3cm}+\int_{\R^n}\bigg[\frac{\partial a(x)}{\partial x^j}\reps(x-y)f(x)-\frac{\partial a(y)}{\partial y^j}\reps(x-y)f(y)\bigg]\,\dint y\\\nonumber
    &\hspace{3cm}+\int_{\R^n}\big(a(x)-a(y)\big)\reps(x-y)\frac{\partial f(y)}{\partial y^j}\,\dint y\bigg\vert^p\dint x\\\nonumber
   &=:\int_{K}\abs{(\Box)_1(x)+(\Box)_2(x)+(\Box)_3(x)}^p\,\dint x.
    \end{align}
    Now we go on estimating the first term on the right hand side. Calculating similarly to \eqref{Eq:K L1} we have
   \begin{equation}
    \begin{split}
        \int_K|(\Box)_1&(x)|^p\,\dint x=\int_K\abs{\int_{\R^n}\frac{\partial a(x)}{\partial x^j}\big(f(y)-f(x)\big)\reps(x-y)^{\frac{1}{p}+\frac{1}{q}}\,\dint y}^p\,\dint x\\[.5\baselineskip]
       &\leq\int_K\bigg[\bigg(\int_{\R^n}\abs{\frac{\partial a(x)}{\partial x^j}\big(f(y)-f(x)\big)}^p\reps(x-y)\,\dint y\bigg)^\frac{1}{p}\bigg(\int_{\R^n}\reps(x-y)\,\dint y\bigg)^\frac{1}{q}\bigg]^p\,\dint x\\[.5\baselineskip]
        &=\int_K\abs{\frac{\partial a(x)}{\partial x^j}}^p\int_{\R^n}\abs{f(y)-f(x)}^p\reps(x-y)\,\dint y\,\dint x\\[.5\baselineskip]&
        \label{Eq:box 1}
        \leq \Lip(a,K)^p\ \eps^p\ \Vert\nabla f\Vert^p_{L^\infty(K')}\ \abs{K}.
    \end{split}
     \end{equation}
    In the same way we estimate the $(\Box)_3$-term by
    \begin{align}
        \nonumber\int_K|(\Box)_3(x)|^p\,\dint x&=\int_K\abs{\int_{\R^n}\big(a(x)-a(y)\big)\ \frac{\partial f(y)}{\partial y^j}\ \reps(x-y)\,\dint y}^p\,\dint x\\[.5\baselineskip]  \label{Eq:box 3}
       &\leq\int_K\int_{\R^n}\abs{a(x)-a(y)}^p\abs{\frac{\partial f(y)}{\partial y^j}}^p\reps(x-y)\,\dint y\,\dint x\\[.5\baselineskip] \nonumber
      &\leq\varepsilon^p\ \Lip(a,K')^p\ \Vert\nabla f\Vert_{L^\infty(K')}^p\ \abs{K}.
    \end{align}
    Finally for $(\Box)_2$ we have
    \begin{align}
        \int_K|(\Box)_2(x)|^p\,\dint x&=\nonumber\int_K\abs{\int_{\R^n}\bigg[\frac{\partial a(x)}{\partial x^j}f(x)-\frac{\partial a(y)}{\partial y^j}f(y)\bigg]\reps(x-y)\,\dint y}^p\,\dint x\\[.5\baselineskip]
        \label{Eq:box 2r}
        &=\int_K\abs{\frac{\partial a(x)}{\partial x^j}f(x)-\int_{\R^n}\frac{\partial a(y)}{\partial y^j}f(y)\reps(x-y)\,\dint y}^p\,\dint x\\[.5\baselineskip] \nonumber
        &=\left\Vert(\partial_ja)f-\big((\partial_ja)f\big)*\reps\right\Vert_{L^p(K)}^p,
    \end{align}
which goes to zero. Indeed $(\partial_ja)f\in L^\infty(K)$ and the conclusion follows from standard properties of convolution (see e.g.\  \cite[Theorem 2.29 (c)]{A:75}).
Putting \eqref{Eq:box 1}, \eqref{Eq:box 3} and \eqref{Eq:box 2r} together via the triangle inequality we 
obtain $K_\eps f\to 0$ in $L^p(K)$ for all test functions $f$. 

 To finish the proof of (i) we still have to deal with the first term on the r.h.s.\ of \eqref{incrediblequatsecond}.
Following the same strategy, we rewrite it as
\begin{align}
    \int_{\R^n}\frac{\partial}{\partial x^j}\Big[\big(a_\varepsilon(x)-a(x)\big)\reps(x-y)\Big]f(y)\,\dint y=:\int_{\R^n}h_\varepsilon(x,y)f(y)\,\dint y=:H_\varepsilon f(x)
\end{align}
Then we have similarly to the above:
    There exists $C>0$ such that
    \begin{equation}\begin{aligned}\label{Eq:H_eps dx bounded}
    &\int_{K}\abs{h_\varepsilon(x,y)}\,\dint x\leq C\quad\text{for all $y\in \R^n$ and all $\varepsilon>0$,}\\
    &\int_{\R^n}\abs{h_\varepsilon(x,y)}\,\dint y\leq C\quad\text{for all $x\in K$ and all $\varepsilon>0$.}\\
    \end{aligned}\end{equation}
  Indeed, writing out $h_\varepsilon$ explicitly we have
  \begin{align}\label{Eq: h_eps explicit}
      h_\varepsilon(x,y)=\frac{\partial}{\partial x^j}\big(a_\varepsilon(x)-a(x)\big)\reps(x-y)+\big(a_\varepsilon(x)-a(x)\big)\frac{\partial}{\partial x^j}\reps(x-y),
  \end{align}
  and 
  using Lemma \ref{Le: a and a_eps}(i), (ii) and (iv) we obtain the estimate
 \begin{align}
      \int_{{K}}\abs{h_\varepsilon(x,y)}\,\dint x
      \leq {\Tilde{C}}  
      + \varepsilon\, C_K\, \frac{1}{\varepsilon}\, \int_K\  \left|\frac{\partial\rho(z)}{\partial z^j}\right|\ \dint z,
\end{align} 
  for some constant
  $\Tilde{C}$, and the same arguments apply to the $y$-integral.

  Now it follows as above that the operators $H_\varepsilon: L^p(\R^n)\to L^p(K)$ are uniformly bounded for $1\leq p<\infty$. 
Hence to prove pointwise convergence in $L^p$ 
we again only have to consider test functions $f$. First we write
\begin{equation}
    \begin{split}
        H_\eps f(x)&=\int_{\R^n}\Big[\frac{\partial}{\partial x^j}\big(a_\eps(x)-a(x)\big)\rho_\eps(x-a)+
        \big(a_\eps(x)-a(x)\big) \frac{\partial}{\partial x^j}\rho_\eps(x-y)\Big]\,f(y)\,\dint y\\
        &=\int_{\R^n} \frac{\partial}{\partial x^j}\big(a_\eps(x)-a(x)\big)\rho_\eps(x-a)f(y)\,\dint y\\
        &\hspace{5cm}
        + \int_{\R^n}  \big(a_\eps(x)-a(x)\big) \rho_\eps(x-y)  \frac{\partial}{\partial y^j}f(y)\,\dint y\\
        &=:(\Box)_4(x)\ +\ (\Box)_5(x).
    \end{split}
\end{equation}
Now we estimate using H\"{o}lder's inequality as in \eqref{Eq:K L1}
\begin{equation}
    \begin{split}
     \int_K |(\Box)_4(x)|^p \,\dint x 
     &\leq \int_K \Big( \int_{\R^n} \Big\vert \frac{\partial}{\partial x^j} \big(a_\eps(x)-a(x)\big)\Big\vert \ \rho_\eps(x-y)^{\frac{1}{p}+\frac{1}{q}}\ \vert f(y)\vert \,{\,\dint y} \Big)^p\,\dint x\\[.5\baselineskip]
     &\leq \Vert f \Vert_{L^\infty(K')}^p \int_K \Big\vert  \frac{\partial}{\partial x^j} \big(a_\eps(x)-a(x)\big)\Big\vert^p \int_{\R^n}\rho_\eps(x-y)\ \,\dint y\ \,\dint x\\[.5\baselineskip]
     &\leq \Vert f \Vert_{L^\infty(K')}^p \ \Vert a_\eps-a\Vert_{W^{1,p}(K')}^p,
    \end{split}
\end{equation}
which goes to zero by Lemma \ref{Le: a and a_eps}(iii). Finally we have 
\begin{equation}
    \begin{split}
          \int_K |(\Box)_5(x)|^p \,\dint x 
          &\leq \int_K  \Big| \big(a_\eps(x)-a(x)\big) \int_{\R^n}  \rho_\eps(x-y) \frac{\partial}{\partial y^j}f(y)
          \,\dint y\Big|^p\,\dint x\\[.5\baselineskip]
          &\leq \Vert \nabla f\Vert _{L^\infty(K')}^p\ \Vert a_\eps-a\Vert_{L^\infty(K)}^p\ \abs{K},
    \end{split}
\end{equation}
which goes to zero thanks to Lemma \ref{Le: a and a_eps}(iv). Summing up we have  
shown that $H_\eps f\to 0$ in $L^p(K)$ for all test functions $f$ and so we obtain (i).

\medskip

To prove (ii) we use the decomposition of the kernel $k_\eps$ given in \eqref{eq:kernel-deriv}, which for $x\in K$ leads to 
\begin{align}
|K_\eps f(x)|
&\leq \int_{\R^n} \left|\frac{\partial a(x)}{\partial x^j}\right|\,\rho_\eps(x-y)\, |f(y)| \,\dint y 
+\int_{\R^n} \left|a(x)-a(y)\right|\,\left|\frac{\partial \rho_\eps(x-y)}{\partial x^j}\right|\, |f(y)|\,\dint y\\[.5\baselineskip]
&\leq \Lip(a,K)\ \Vert f\Vert_{L^\infty(K')}\ + \ 
\varepsilon\ \Lip(a,K')\ \frac{1}{\varepsilon}\ \Vert \nabla\rho\Vert_{L^\infty(K')} \ \Vert f\Vert_{L^\infty(K')}.
\end{align}
To derive the desired $L^\infty$-bound it only remains to take care of the term $H_\varepsilon f$ on the r.h.s. of \eqref{incrediblequatsecond}. Similarly to the above we have for $x\in K$
\begin{align}\nonumber
    \abs{H_\varepsilon f(x)}&\leq\int_{\R^n}\abs{\frac{\partial}{\partial x^j}\Big(a_\varepsilon(x)-a(x)\Big)}\ \reps(x-y)\ \abs{f(y)}\,\dint y\\&\hspace{5cm}+\int_{\R^n}\abs{a_\varepsilon(x)-a(x)}\abs{\frac{\partial \reps(x-y)}{\partial x^j}}\abs{f(y)}\,\dint y \\[.5\baselineskip]\nonumber
    &\leq {\tilde C\ 
    \Vert f\Vert_{L^\infty(K')}}\ +\ C_K\, \varepsilon\ \frac{1}{\varepsilon}\ \Vert\nabla\rho\Vert_{L^\infty(K')}\ \Vert f\Vert_{L^\infty(K')},
\end{align}
where we used \Cref{Le: a and a_eps} {(i), (ii) and (iv)} in the last inequality.
\end{proof}

With an estimate of the form of \eqref{eq:thm2.4} in mind, we also insert vector fields into the respective Ricci-terms. We now do so in a form directly usable in the course of our main proofs, cf.\ \Cref{Lem:(iv) of our Hawking theorem} below. The proof strategy is the same as in \Cref{order1Fridrichs}.

\begin{cor}\label{Le:Box2 of L4.8}
    Let $X, Y\in\mathfrak{X}(M)$. Then for all $K\Subset M$ and any $p\in [1,\infty)$ we have
\begin{align}\label{Eq: ric lem4.8}
(\Ric[g]\star_M\reps)(X,Y)-(\Ric[g](X,Y))\star_M\reps\rightarrow0\quad\text{in }L^p(K),
\end{align}
as $\eps\rightarrow 0$.
\end{cor}
\begin{proof}
As in the discussion following \eqref{incrediblequat}, we determine the relevant terms in the Ricci tensor to be locally of the form $\partial (af)$, where $a$ is Lipschitz and $f\in L^\infty$. This precise structure is not
required for the current proof, so we simply write $f\in L^\infty$ in  place of $af$. We use the letter $a$ instead to represent terms of the form $X^iY^j$ by a smooth function in local coordinates. Then proving \eqref{Eq: ric lem4.8} amounts to showing that
 \begin{align}\label{Eq: ric lem4.8 rewritten}
     [(\partial_j f)*\reps]a-[(\partial_j f)a]*\reps\rightarrow0\quad\text{in }L^p(K).
 \end{align}
Applying \Cref{order1Fridrichs} (i) to the constant net $a_\eps=a$ we obtain that
\begin{align}
(\partial_j a)(f*\rho_\eps) + [\partial_j(f*\rho_\eps)]a - [(\partial_j a)f]*\rho_\eps - [(\partial_j f)a]*\rho_\eps \rightarrow0\quad\text{in }L^p(K).
\end{align}
In this expression, $(\partial_j a)(f*\rho_\eps) - ((\partial_j a)f)*\rho_\eps \to 0$ in $L^p(K)$, so \eqref{Eq: ric lem4.8 rewritten} follows.
\end{proof}

Finally, we note the following standard result from functional analysis, which was key in the proof of \Cref{order1Fridrichs}.

\begin{lem}\label{functionalanalysislemma}
   Let $K\Subset\R^n$ and $1\leq p<\infty$. Suppose $T_n:\ L^p(\R^n)\rightarrow L^p(K)$ is a sequence of uniformly bounded linear operators such that $T_n\rightarrow0$ pointwise on $C^\infty_c(\R^n)$. Then $T_nf \to 0$ in $L^p(K)$ for each $f\in L^p(\R^n)$.
\end{lem}

\section{Mean curvature for Lipschitz metrics}\label{sec:meancurvature}

Regarding the `initial condition' of Hawking's singularity theorem we need to define what it means to bound the mean curvature of a spacelike hypersurface $\Sigma$. To do so we shall use a `thickening' of the Lebesgue zero set $\Sigma$, to properly extend the classical notion, as follows.

Let $(M,g)$ be a spacetime with a Lipschitz metric tensor $g$ and let $X\in\mathfrak{X}(M)$ be smooth and timelike. 
Let $\Phi:\mathcal O\rightarrow M$ be the \textit{flow-out} of $\Sigma$ along $X$ (cf.\ e.g.\ \cite[Thm. 9.20]{Lee:00}), where $\mathcal O\subseteq \Sigma\times\R$ is open.
We say that an open set $\mathcal{A}_\Sigma\subseteq  \mathcal{O}$ is \textit{admissible} if $\mathcal{A}_\Sigma$ is an open neighbourhood of $\Sigma\times\{0\}$, $\Phi|_{\mathcal{A}_\Sigma}$ is a diffeomorphism onto its image and the induced metric on each 
hypersurface $\Sigma_t = \Phi((\Sigma \times \{t\}) \cap \mathcal{A}_\Sigma)$ is Riemannian.  
We then also call $A_\Sigma=\Phi(\mathcal A_\Sigma)$ admissible and note that such sets exist by the flow-out theorem
since $\Sigma$ is spacelike.

In what follows we 
denote by $N$ the unique future directed unit vector field which is normal to the leaves $\Sigma_t$ and note that 
$N$ is locally Lipschitz.
We call a vector field $Y\in \mathfrak{X}(A_\Sigma)$ \emph{tangent} to $A_\Sigma$ and write $Y\in\mathfrak X(A_\Sigma)^\top$ if $Y(p)\in T_p\Sigma_t$ for any
$p\in \Sigma_t$ for all $t$. 

With these conventions, given a smooth timelike vector field $X$  we now define the \emph{$X$-slab mean curvature} $\mathcal{H}^X_{\mathcal A_\Sigma}[g]\in L^\infty_\loc( A_\Sigma)$ of $\Sigma$ on an admissible open set $A_\Sigma$ as 
\begin{align}
    \mathfrak{X}(A_\Sigma)^\top \times \mathfrak{X}(A_\Sigma)^\top \ni (Y,Z) \mapsto -\mathrm{tr}_g\left(g(\nabla_YZ,N)\right).
\end{align}
To be precise, by $\mathrm{tr}_g$ we mean here the $(n-1)$-dimensional metric trace along the spacelike slices of $A_\Sigma$. Thus if $X_1,\dots,X_{n-1}$ form a local frame of tangential vector fields and $G^{ij}$ denotes the
inverse of the matrix $(g(X_i,X_j))_{i,j=1}^{n-1}$, then on the common domain of the $X_i$'s we have
\begin{align}\label{Eq:meanslab}
\mathcal{H}^X_{\mathcal A_\Sigma}[g] = -\sum_{i,j=1}^{n-1} G^{ij} g(\nabla_{X_i}X_j,N).
\end{align}
We will then say that the $X$-slab mean curvature of $\Sigma$ is bounded above by $b \in \R$ if there exists an admissible set $\mathcal A_\Sigma$ such that \eqref{Eq:meanslab} is locally essentially bounded, i.e.
\begin{align}\label{Eq: X-slab bounds}
\esssup \mathcal{H}_{\mathcal A_\Sigma}^X[g] < b \qquad 
\text{ on } A_\Sigma.
\end{align}
In this case, we simply write  
$\mathcal{H}^X[g] < b$.
Note that the property of being admissible is closed under intersections. Also, if there is an $\mathcal A_\Sigma$ such that \eqref{Eq: X-slab bounds} holds, then for any other admissible open set $\mathcal B_\Sigma$  
the bound \eqref{Eq: X-slab bounds} also holds for $\mathcal B_\Sigma \cap \mathcal A_\Sigma$.
We now show that this notion of mean curvature bounds is independent of the choice of $X$.
\begin{lem}[Independence of $X$]\label{Lem: X indep}
Let $(M,g)$ be a spacetime with a locally Lipschitz metric $g$ and let $\Sigma \subseteq  M$ be a smooth spacelike hypersurface. Assume that $\mathcal{H}^X[g]<b$ for some timelike $X\in\mathfrak{X}(M)$. Then given any other timelike $Y\in\mathfrak{X}(M)$ it also holds that $\mathcal{H}^Y[g]<b$. 
\end{lem}
\begin{proof} 
We fix a local frame $V_1,\dots,V_{n-1}\in\mathfrak{X}(\Sigma)$ and define $X_1,\dots,X_{n-1} \in \mathfrak{X}(A_\Sigma)^\top$ and $Y_1,\dots,Y_{n-1}\in\mathfrak{X}(B_\Sigma)^\top$ by the push-forward of 
this frame under the diffeomorphisms $\Phi^X$ and $\Phi^Y$, i.e.,
\begin{align}
    X_{i} \circ \Phi^X(z, t) :=T_{(z,t)}\Phi^X\left(V_i(z), 0 \right),\qquad Y_{i}\circ \Phi^Y(z, t):=T_{(z,t)}\Phi^Y\left(V_i(z),0 \right).
\end{align}
Then $(X_i)_t, (Y_i)_t\rightarrow V_{i}$ in $C^\infty$ as $t\rightarrow0$.
Denoting by $N^X$ and $N^Y$ the corresponding future directed unit normal fields, we note that $(N^X)_t \in\mathfrak{X}(\Sigma^X_t)^\perp$ and $(N^Y)_t \in\mathfrak{X}(\Sigma^Y_t)^\perp$ converge locally uniformly to the future-directed unit normal of $\Sigma$  as $t\rightarrow0$.

Let $\Sigma$ be compact for the moment. From the assumption $\mathcal{H}^X[g]<b$ in terms of the essential supremum there exists $\eta>0$ such that $\smash{\mathcal{H}_{\mathcal A_\Sigma}^X[g]<b-\eta}$ a.e.\ for some admissible $\mathcal A_\Sigma$. 
By the above convergence properties and 
the almost everywhere local boundedness of $\Gamma^i_{jk}[g]$, 
there exists an open neighbourhood $\smash{Z_\Sigma\subseteq \Phi^X(\mathcal A_\Sigma)\cap\Phi^Y(\mathcal B_\Sigma)}$ containing $\Sigma$ such that
\begin{align}
    \smash{\left\Vert\mathcal{H}_{\mathcal A_\Sigma}^X[g]-\mathcal{H}_{\mathcal B_\Sigma}^Y[g]\right\Vert_{L^\infty_\loc(Z_\Sigma)}<\frac{\eta}{2}.}
\end{align}
Consequently,
\begin{align}   
    \smash{\mathcal{H}_{\mathcal B_\Sigma}^Y[g]<
    b-\frac{\eta}{2}\qquad\mbox{a.e.\ on } Z_\Sigma}.
\end{align}
We construct an admissible open set from $Z_\Sigma$ as $\mathcal U_\Sigma:=
(\Phi^Y)^{-1}(Z_\Sigma)$. Since $\mathcal{H}^Y_{\mathcal{B}_\Sigma}[g]$ and $\smash{\mathcal{H}^Y_{\mathcal{U}_\Sigma}[g]}$ coincide on $U_\Sigma \subseteq  Z_\Sigma$, we have that $\smash{\mathcal{H}^Y_{\mathcal U_\Sigma}[g]<b -\frac{\eta}{2}}$ a.e.\ on $U_\Sigma$ and thereby $\mathcal{H}^Y <b$ in the sense of \eqref{Eq: X-slab bounds}. \medskip

For the general case consider an exhaustion of $\Sigma$ by compact sets $\{K_i\}_{i\in\N}$. Notice that on each $K_i$, there exists an $\varepsilon_i>0$ such that $\Phi^X(K_i \times (-\varepsilon_i, \varepsilon_i))$ and $\Phi^Y(K_i \times (-\varepsilon_i, \varepsilon_i))$ are contained in $\Phi^X(\mathcal{A}_\Sigma)$ and $\Phi^Y(\mathcal{B}_\Sigma)$, respectively. By the preceding argument, on each $K_i$ it holds that $\mathcal{H}^Y_{\mathcal{B}_\Sigma}[g]<b-\frac{\eta}{2}$ almost everywhere on an open subset $Z^i_\Sigma \subseteq  M$ containing $K_i$. Hence we can choose a smooth function $\delta:\Sigma\rightarrow\R^+$ such that, for each $i$, $\smash{\delta\vert_{K_i}<\varepsilon_i}$ 
and $\smash{\Phi^Y_\delta(K_i)}$ is contained in $Z^i_\Sigma$, where $\Phi^Y_\delta:= \Phi^Y(\cdot, \delta(\cdot))$. Then 
on the admissible open set $\smash{\mathcal{C}_\Sigma:={(\Phi^Y)}^{-1}(\bigcup_{i=1}^\infty \Phi^Y_\delta(K_i))}$ we obtain 
\begin{align}   
\mathcal H^Y_{\mathcal C_\Sigma}[g] <b -\frac{\eta}{2}\qquad \text{a.e.\ on } C_\Sigma,
\end{align}
which shows that $\mathcal{H}^Y[g]<b$, as claimed.
\end{proof}

Thanks to \Cref{Lem: X indep} we can now define \emph{bounds for the slab mean curvature of $\Sigma$} independent of the choice of a timelike vector field
$X\in\mathfrak{X}(M)$.

\begin{defi}[Mean curvature bounds]\label{De:mean curvature bounds}
    Let $(M, g)$ be a spacetime with Lipschitz metric $g$. Fix a constant $b\in\R$ and let $\Sigma$ be a smooth spacelike hypersurface. We say that the slab mean curvature of $\Sigma$ is bounded above by $b$ and write
    \begin{align*}
        \mathcal{H}[g]<b
    \end{align*}
    if there exists a timelike $X\in\mathfrak{X}(M)$ such that $\mathcal{H}^X[g]<b$ in the sense of \eqref{Eq: X-slab bounds}.
    \end{defi}

To justify the definition above, we next show that bounds for the slab mean curvature are equivalent to bounds for the usual mean curvature \emph{if $g$ is smooth}.
\begin{lem}[Equivalence of mean curvature bounds for $g$ smooth]\label{lemma:eq_meancurvature}
    Let $(M, g)$ be a smooth spacetime. Let $\Sigma$ be a smooth spacelike hypersurface with future-directed unit normal vector field $\vec{n}$, and let $H[g]$ be the mean curvature of $\Sigma$ associated with $\vec{n}$. For  
    any $b\in\R$, the following statements are equivalent:
    \begin{enumerate}
        \item $H[g] < b $ on $\Sigma$.
        \item $\mathcal H[g]<b$.
    \end{enumerate}
\end{lem}
\begin{proof} Note that for $N$ as defined above we have $N|_\Sigma = \vec{n}$.
The implication (ii) $\Rightarrow$ (i) is immediate since $\mathcal{H}^X_{\mathcal{A}_\Sigma}$ and $H$ agree on $\Sigma$. Conversely, (i) $\Rightarrow$ (ii) follows by continuity since for any timelike $X\in\mathfrak X(M)$, the trace $-\smash{\mathrm{tr}_g g\left(\nabla_{\,.\,}\,.\,,N^X\right)}$ is smooth and restricts to $H[g]$ on $\Sigma$.
\end{proof}
The following final result of this section will allow us to infer mean curvature bounds on regularisations of $g$ from those on $g$ in the sense of the previous definition.

\begin{lem}[Local convergence of the mean curvature]\label{Le: ConvMeanCurv}
    Let $(M,g)$ be a spacetime with a locally Lipschitz metric $g$ and let $\Sigma$ be a smooth spacelike hypersurface.
   Let $\tilde\Sigma$ be open and relatively compact in $\Sigma$ and let $\smash{g_k:= \check g_{\eps_k}}$ be as in \Cref{Le:approximating metrics} (iv). Finally, let $X\in \mathfrak{X}(M)$ be $g$-timelike.
    Then there exists some $k_0\in \N$ and an open neighbourhood $A_{\tilde\Sigma}$ of $\tilde\Sigma$ in $M$ that is admissible with respect to $X$ for $g$ 
    and for each $g_k$ with $k\ge k_0$. Moreover, 
    \begin{align}\label{Eq:mean curv conv}
        \smash{\big\Vert\mathcal{H}^X[g_k] -\mathcal{H}^X[g] \star_M \rho_{\varepsilon_k}\big\Vert_{L^\infty(A_{\tilde\Sigma})}\rightarrow0\quad\text{ as } k \to \infty.}
    \end{align}
In particular, if $\mathcal{H}[g]<b$ for some $b\in \R$, then the $g_k$-mean curvature of $\tilde\Sigma$ satisfies $H_{\tilde\Sigma}[g_k]<b$ for $k$ large.
\end{lem}

\begin{proof}
Let $A_\Sigma$ be admissible for $\Sigma$, $X$ and $g$ and choose an open and relatively compact neighbourhood $A_{\tilde\Sigma} \subseteq A_\Sigma$ of $\tilde\Sigma$. Then since $g_k\to g$ uniformly on $A_{\tilde\Sigma}$, there exists $k_0$ such that the leaves $\tilde\Sigma_t$ are $g_k$-spacelike, and thereby $A_{\tilde\Sigma}$ is 
admissible for $\tilde\Sigma$, $X$ and $g_k$, for each $k\ge k_0$.

Dropping the subscript indicating the admissible open set for the $X$-slab mean curvature, we have
    \begin{align}\label{Eq:conv mean curv rewritten}
        \mathcal{H}^X[g_k]-\mathcal{H}^X[g]\star_M\rho_{\eps_k}=\left(\mathcal{H}^X[g_k]-\mathcal{H}^X[g_{\eps_k}]\right)+\left(\mathcal{H}^X[g_{\eps_k}]-\mathcal{H}^X[g]\star_M\rho_{\eps_k}\right).
    \end{align}
    Due to \Cref{Eq:meanslab} we can schematically write the local form of the highest order terms of $\mathcal{H}^X[g]$ in the form
    \begin{align}\label{Eq:meanschematic}
        \mathcal{H}^X[g]\sim g(\partial g)g^{-1}\eta_i \xi_j\xi_k\sim g(\partial g),
    \end{align}   
    where $\eta_i$ represents local components of $N^X$ and  $\xi_j,\xi_k$ stand for (local derivatives of) $X_i$. Here, 
    $g^{-1}$ and $\eta_i$ are  Lipschitz and the $\xi_j,\xi_k$ are smooth, hence can be notationally suppressed.
    Combining \eqref{Eq:conv mean curv rewritten} and \eqref{Eq:meanschematic} we see that \Cref{Le:approximating metrics} directly gives  convergence (even in $C^\infty_{\mathrm{loc}}(M)$) for the first term on the right-hand side of \eqref{Eq:conv mean curv rewritten}. For the second one, notice that (in a local chart)
    \begin{align*}
        g_\eps(\partial g_\eps)-\big(g(\partial g)\big)*\reps=(g*\reps)\big((\partial g)*\reps\big)-\big(g(\partial g)\big)*\reps
    \end{align*}
    has the same form as \cite[Eq.\ (6) of Lemma 3.2]{KSSV:15} since $g$ is Lipschitz and hence $\partial g\in L^\infty_{\mathrm{loc}}(M)$, so it goes to $0$ locally uniformly, as claimed.

    The final claim now follows from Lemmas \ref{Lem: X indep} and \ref{lemma:eq_meancurvature}.
\end{proof}

\begin{rem}[Comparison with synthetic notions of mean curvature]\label{Rem: mean curv comp}
    In the past few  
    years various synthetic definitions of mean curvature and mean curvature bounds have been proposed,
    based on the \emph{needle decomposition}, a tool used to localize (timelike) lower Ricci curvature bounds in the synthetic setting \cite{Ket:20, 
    CM:20, BMcC:23, Ket:23, Bra:24}.
    This notion of localization stems 
    from convex geometry and serves the purpose of 
    reducing 
    a multi-dimensional problem to a one-dimensional one, in order to derive functional and geometric inequalities, in particular 
    lower Ricci curvature bounds  \cite{K:17, CFMcC:02, CM:17II, CM:20}.
    
    In summary, this construction decomposes the ambient measured Lorentzian pre-length space $(X,\textsf{d},\leq,\ll,\tau,\m)$ into maximisers  $X_\alpha$, where $\alpha$ is an index that can be identified with points in the Cauchy hypersurface $\Sigma$, up to a negligible set. Recall that in the Lipschitz case maximisers can be parametrised to be $C^{1,1}$-geodesics in the Filippov sense, see Remark \ref{rem:geomax}(ii).
    The so-called \emph{needles} $X_\alpha$ are then viewed as one-dimensional metric measure spaces $(X_\alpha,\textsf{d}_\alpha,\m_\alpha)$ that inherit the curvature properties of the ambient space, cf. \cite[Thm. 4.17]{CM:20} and \cite[Thms. 6.37 and A.5]{BMcC:23}. As such, these curvature properties yield the absolute continuity of $\m_\alpha=h_\alpha\mathcal{L}^1\vert_{[0,L_\tau(X_\alpha)]}$ \cite[Thm. A.2]{CM:21}. In the smooth context it follows directly from the definitions that
    \begin{align}\label{Eq: halpha}
        h_\alpha(t)=\psi\cdot\det D\Phi_{(\alpha,t)}\vert_{T_\alpha\Sigma}
    \end{align}
    where $\Phi(\alpha,t)=\exp_\alpha(-t\nabla\tau_\Sigma(\alpha))$ is the normal exponential map and $\psi$ is a renormalization factor \cite[Rem. 5.4]{CM:20}.    

   Synthetic mean curvature bounds are then defined via a second-order Taylor expansion of the volume of the region that in the smooth case is spanned by the evolution of $\Sigma$ under the map $\Phi$, hence involves $h_\alpha$ directly (\cite[Def. 5.2]{CM:20}). In particular, just like the notion put forth in \Cref{De:mean curvature bounds},  mean curvature bounds in this synthetic sense reproduce the classical bounds in the smooth case (\cite[Rem. 5.4]{CM:20}).  
   To compare these notions for Lipschitz Lorentzian metrics
   it is necessary to study what the right-hand side of \eqref{Eq: halpha} means in this regularity, where the exponential map is no longer available.
    In addition, it requires a study of the needle decomposition in the context of low-regularity metrics,
    which is deferred to future work.
    \end{rem}

\section{Hawking's singularity theorem for Lipschitz metrics}\label{sec:main}

In this section we formulate and prove our main results. As usual we give two versions of the Hawking theorem, the first one providing a global bound on $\tau_\Sigma(p) = \sup_{q\in \Sigma}\tau(q,p)$ (cf.\ \eqref{eq:subset time sep}) for a Cauchy surface $\Sigma$, while the second one avoids global hyperbolicity.

\subsection{The globally hyperbolic case}
In this section we will prove the following general version of Hawking's singularity theorem:

\begin{thm}[$C^{0,1}$-Hawking singularity theorem, I]\label{Th: Hawking gh case}
   Let $(M,g)$ be a  spacetime with a locally Lipschitz metric $g$ such that:
    \begin{itemize}
        \label{Th: Hawking gh case (i)}
        \item[(i)] There exists  
        $\rho\in \R$
        such that $\Ric_g(X,X)\geq -(n-1)\rho\, g(X,X)$ in the distributional sense for all timelike $X\in\mathfrak{X}(M)$.
        \label{Th: Hawking gh case (ii)}
        \item[(ii)] There is a smooth spacelike Cauchy hypersurface $\Sigma$ with  $\mathcal{H}[g]<\beta<0$.
    \end{itemize}
    Then we have     
    \begin{equation}\label{eq:Hawking-bounds-cases}
      \sup_M\tau_\Sigma\leq  \begin{cases}
        \frac{n-1}{\abs{\beta}} &\text{if} \;\rho=0\\
          \frac{1}{\sqrt{|\rho|}}\coth^{-1}\Big(\frac{|\beta|}{(n-1)\sqrt{|\rho|}}\Big) &\text{if} \;\rho<0, \text{ provided } |\beta|>(n-1)\sqrt{|\rho|}, \\
         \frac{1}{\sqrt{\rho}}\cot^{-1}\Big(\frac{|\beta|}{(n-1)\sqrt{\rho}}\Big) &\text{if} \;\rho>0.
       \end{cases}
   \end{equation}
\end{thm} 

\begin{rem} \ 
\begin{itemize}
\item[(i)] The case $\rho=0$ in \eqref{eq:Hawking-bounds-cases} corresponds to the classical Hawking theorem (cf.\ \cite[Thm.\ 14.55A]{ONe:83}). Versions of the theorem  with $\rho\not=0$ have appeared e.g.\ in \cite{Bor:94,AG:02} and in low regularity in \cite{Gra:16,CM:20}.
\item[(ii)] Denoting the right hand side of \eqref{eq:Hawking-bounds-cases} by $\alpha(\beta,\rho)$ (cf.\ \eqref{eq:alpha-cases} below), $\alpha(\beta,\rho)$ is calculated such that $K(\beta,\alpha(\beta,\rho),\rho)=0$, see \eqref{eq:K-def}. The assumption on $\beta$ and $\rho$ in
the case $\rho<0$ ensures that the argument lies within the domain of $\coth^{-1}$.
\item[(iii)] In the case $\rho>0$ the expression on the right hand side of \eqref{eq:Hawking-bounds-cases} is always strictly smaller than $\frac{\pi}{\sqrt{\rho}}$ and in fact smaller than $\frac{\sqrt{\rho}}{2\pi}$. This is in accordance with the observations that (in the smooth setting), (a) lower timelike
Ricci curvature bounded below by $(n-1)\rho>0$ forces any unit speed timelike geodesic of length $\ge \frac{\pi}{\sqrt{\rho}}$ to contain conjugate points (cf.\ \cite[Lem.\ 10.23]{ONe:83}) and (b) any hypersurface with negative mean curvature (i.e. an initial focussing) develops conjugate points before $\frac{\sqrt{\rho}}{2 \pi}$. 
\end{itemize}    
\end{rem}

Throughout this section, we will generally assume---unless explicitly stated otherwise---that $(M,g)$ is globally hyperbolic and that $g$ is locally Lipschitz. To prove \Cref{Th: Hawking gh case} some preparations are required.

\begin{lem}[Existence of maximisers]\label{Le:properties of sigma}
    Let $\Sigma\subseteq M$ be a Cauchy hypersurface of a globally hyperbolic spacetime $(M,g)$ with a locally Lipschitz metric $g$. Then 
    \begin{enumerate}
        \item For any compact set $K\Subset J^+(\Sigma)$ the set $J^-(K)\cap J^+(\Sigma)$ is compact as well.
        \item  For any $q\in J^+(\Sigma)$ there exists a maximising causal curve $\gamma$ from $\Sigma$ to $q$. 
    \end{enumerate}
\end{lem}
Recall that by Remark \ref{rem:geomax}(ii) the maximiser $\gamma$ in (ii) above, when parametrised w.r.t.\ $g$-arclength, is a $C^{1,1}$-geodesic in the sense of Filippov.
 
\begin{proof}
(i) $\Sigma$ is a closed acausal topological hypersurface. As noted in Remark \ref{rem:cs}, $C_p:=\{v\in T_pM\,\backslash\,\{0\}:g(v,v)\leq0, v\mbox{ future directed}\}$ 
defines a proper cone structure. So, using $J^+(\Sigma)=D^+(\Sigma)$, we obtain by \cite[Thm.\ 2.44]{Min:19a} that $J^-(K)\cap J^+(\Sigma)$ is compact.
\medskip

(ii) By Remark \ref{rem:lls}, $M$ is a globally hyperbolic Lorentzian length space and so by \cite[Thm.\ 3.28]{KS:18}, $\tau$ is finite and continuous.  Thus $\tau(\cdot,q)$ has a maximum on $J^-(q)\cap \Sigma$, which is compact by (i). 
 Hence there is a $p\in J^-(q)\cap \Sigma$ with $\tau_\Sigma(q)=\tau(p,q)$. By the Avez-Seifert theorem (\cite[Prop.\ 6.4]{Sae:16} or \cite[Thm.\ 3.30]{KS:18}) there exists a maximiser from $p$ to $q$, hence from $\Sigma$ to $q$.
\end{proof}

\begin{lem}[Continuity of $\tau_\Sigma$]\label{Le:continuity of tau_simga}
    Let $(M,g)$ be a globally hyperbolic spacetime with a continuous metric $g$ and Cauchy surface $\Sigma$. Then $\tau_\Sigma$ is continuous.
\end{lem}

\begin{proof}
    By definition, $\tau_\Sigma$ is a supremum of continuous functions hence it is lower semicontinuous and we only need to show upper semicontinuity. 
         
    Suppose to the contrary that 
    there is a sequence 
    $q_n\to q$ and $\delta>0$ such that $\tau_\Sigma(q_n)\geq\tau_\Sigma(q)+\delta$. 
    Let $z\in I^+(q)$, then w.l.o.g., $q_n, q\in J^-(z)$. By (the proof of) Lemma \ref{Le:properties of sigma}(ii) for all $n$ there is $p_n\in \Sigma\cap J^-(z)$ with $\tau_\Sigma(q_n)=\tau(p_n,q_n)$. By compactness, cf.\ Lemma \ref{Le:properties of sigma}(i), we may assume the $p_n$ to converge to some $p\in\Sigma\cap J^-(z)$ and we have
    \begin{equation}
    \tau(p_n,q_n)=\tau_\Sigma(q_n)\geq\tau_\Sigma(q)+\delta\geq\tau(p,q)+\delta
    \end{equation}
    But by continuity $\tau(p_n,q_n)\to\tau(p,q)$, which is a contradiction.
\end{proof}

In the following we will make use of the smooth metrics $\check g_\eps$ approximating the rough metric $g$ `from the inside' as introduced in \Cref{Le:approximating metrics}. In particular, we use a monotone sequence $\check g_{\varepsilon_k}$ as in item (iv) of that lemma, and for convenience we set 
\begin{equation}\label{eq:tauk}
     g_k:=\check g_{\varepsilon_k},\quad \tau_k:=\tau_{\check g_{\eps_k}}, \quad\text{and}\quad \tau_{\Sigma,k}:=\tau_{\Sigma,\check g_{\eps_k}}.
\end{equation}
Since the $g_k$ have narrower light cones 
than $g$, each $g_k$ is globally hyperbolic itself, and any Cauchy surface for $g$ is also a Cauchy surface for each $g_k$.
\medskip 

\begin{lem}[Convergence of time-separations]\label{Le:uniform convergence of tau}
Let $(M,g)$ be a globally hyperbolic spacetime with locally Lipschitz metric $g$ and Cauchy surface $\Sigma$. Then we have
\begin{enumerate}
        \item $\tau_{k}\rightarrow\tau$ locally uniformly, and
        \item $\tau_{\Sigma,k}\rightarrow\tau_\Sigma$ locally uniformly.
    \end{enumerate}
\end{lem}  
    
\begin{rem} In \cite[Prop. A.2]{McCS:22} statement (i) is proven
    for continuous, causally plain and strongly causal metrics $g$, but for monotonously approximating metrics with \emph{wider} light cones, i.e., 
    $\tau_{\hat g_{\varepsilon_k}}\to \tau$ locally uniformly. Observe that their proof does not work in our case since it relies on the property $\tau(p, q)\leq \tau_{\hat g_{\varepsilon_{k+1}}}(p,q)\leq \tau_{\hat g_{\varepsilon_k}}(p,q)$, which is reversed for $\tau_k:=\tau_{\check g_{\eps_k}}$, see \eqref{Eq:tau_k<tau_k+1<tau} below. 
\end{rem}

To establish Lemma \ref{Le:uniform convergence of tau} we shall require the following result, which is the analogue of \cite[Lem. A.1 (iii) and (iv)]{McCS:22} for metrics $g_k=\check g_{\varepsilon_k}$ with narrower light-cones and which can be proven in full analogy.

\begin{lem}\label{Le:thinner lightcones} Let $(M,g)$ be a continuous spacetime and let $g_k$ be as in \eqref{eq:tauk}. Then 
the sequence $\eps_k\searrow 0$ can be chosen in such a way that,
for all $g_k$-causal $X\in TM$:
    \begin{enumerate}
        \item [(i)]$-g_k(X,X)<-g(X,X)$, and 
        \item [(ii)]$-g_k(X,X)\leq-g_{k+1}(X,X)$.
    \end{enumerate}
\end{lem}

    \begin{proof}[Proof of \Cref{Le:uniform convergence of tau}]
    (i) 
    Let $p,q$ in some $K\Subset M$ and observe first that 
    \begin{align}\label{Eq:tau_k<tau_k+1<tau}
        \tau_k(p,q)\leq\tau_{k+1}(p,q)\leq\tau(p,q).
    \end{align}
    Indeed any $g_k$-causal $\gamma$ from $p$ to $q$ is also $g_{k+1}$- and $g$-causal. Then from  \Cref{Le:thinner lightcones}  we have
   $L_{g_k}(\gamma)\leq  L_{g_{k+1}}(\gamma) \leq L_g(\gamma)$ and so 
        \begin{align}
            \label{eq:useless}\tau_k(p,q)=\sup_{\gamma\,g_k\text{-causal}} L_{g_k}(\gamma)\leq\sup_{\gamma\,g_{k+1}\text{-causal}}L_{g_{k+1}}(\gamma)\leq\sup_{\gamma\,g\text{-causal}}L_g(\gamma)=\tau(p,q).
        \end{align}

    We next prove pointwise convergence of $\tau_k$ to $\tau$, which together with monotonicity \eqref{Eq:tau_k<tau_k+1<tau} gives the claim by Dini's theorem (e.g.\ \cite[Thm.\ 2.66]{AB:06}).
    In case $\tau(p,q)=0$, again 
    by \eqref{Eq:tau_k<tau_k+1<tau} all $\tau_k(p,q)$ vanish, and so we only need to consider $\tau(p,q)>0$. 
    By the Avez-Seifert theorem there is a $g$-maximiser $\gamma$ from $p$ to $q$, which when parametrised to $g$-unit speed is a $C^{1,1}$-curve (cf.\ Remark \ref{rem:geomax}(ii)), implying that $\gamma$ is $g_k$-timelike for $k$ large enough,
    and we have $L_{g_k}(\gamma)\rightarrow L_g(\gamma)$. Hence for $\delta >0$ and $k$ large enough
    \begin{align}\label{Eq:tau pointwise}
            \tau_k(p,q) \geq L_{g_k}(\gamma)\geq L_g(\gamma)-\delta=\tau(p,q)-\delta,
        \end{align}
    and convergence follows by combining this estimate with \eqref{Eq:tau_k<tau_k+1<tau}.

    (ii) As before, since $g_k$-causal curves are $g_{k+1}$- and $g$-causal we have monotonicity, i.e., $\tau_{\Sigma,k}\leq\tau_{\Sigma,k+1}\leq\tau_\Sigma$ and again it suffices to show $\tau_{\Sigma,k}(q)\to\tau_\Sigma(q)$ for all $q\in K\Subset M$. Given such $q$ by (the proof of) Lemma \ref{Le:properties of sigma}(ii) there is $p\in\Sigma$ with $\tau_\Sigma(q)=\tau(p,q)$ and since by (i) the latter expression equals $\lim_{k\rightarrow\infty}\tau_k(p,q)$ we have for any $\delta>0$ and all $k$ large enough
    \begin{equation}
        \tau_\Sigma(q)\geq\tau_{\Sigma,k}(q)\geq\tau_k(p,q)\geq\tau(p,q)-\delta=\tau_\Sigma(q)-\delta
    \end{equation}
    and we are done.
\end{proof}

We next employ the regularisation result \Cref{Pr:Ricci Lp-convergence}(ii) to derive a uniform lower bound on the Ricci tensor $\Ric[g_k]$ of the regularised metrics from the distributional strong energy condition on  $\Ric[g]$.

\begin{prop}[Lower uniform bound on ${\Ric[g_k]}$]\label{Pr:(i) for our Hawking theorem}
    Let $(M,g)$ be a spacetime with a locally Lipschitz metric $g$. Suppose that there is some $C\in\R$ such that 
    \begin{equation}\label{eq:ric_pos_assumption}
        \Ric[g](X,X)\geq C \cdot g(X,X)\quad\text{distributionally for all $g$-timelike $X\in\mathfrak{X}(M)$.}
    \end{equation}
    Then for any compact $K\Subset M$ there is $\tilde C\in\R$ such that for all $\eta < 0$ 
    there exists some $k_0\in \N$ such that, for all $k\ge k_0$ and all $X\in TM|_K$ with 
    $\|X\|_h \le D$  and $g(X,X)\le \eta$ we have
    \begin{align}\label{Eq: Ric>-delta}
       \Ric[g_k](X,X)\geq \tilde C\cdot g_k(X,X).
    \end{align}
\end{prop}
\begin{proof}
We first note that it suffices to show the claim for $g_{\eps_k}$ instead of $g_k=(\check g_{\eps_k})$. Indeed, suppose that we already know \eqref{Eq: Ric>-delta} for $g_{\eps_k}$ instead of $g_k$. Then due to \Cref{Le:approximating metrics}, for any $\delta>0$ and $k$ sufficiently large we
have 
\begin{align*}
\Ric[g_k](X,X) - \tilde C\cdot g_k(X,X) &= (\Ric[g_k]-\Ric[g_{\eps_k}])(X,X) \\
&+   (\Ric[g_{\eps_k}](X,X) - \tilde Cg_{\eps_k}(X,X)) + \tilde C(g_{\eps_k}(X,X)-g_k(X,X)) \\
&\ge -\delta \ge  \frac{-\delta}{\eta} g(X,X) \ge \frac{-2\delta}{\eta} g_k(X,X),
\end{align*}
which gives \eqref{Eq: Ric>-delta} with $\tilde C - \frac{2\delta}{\eta}$ instead of $\tilde C$. For the remainder of this proof we therefore may assume that $g_k\equiv g_{\eps_k}$.

Next, since the claim is local, we may suppose that $M=\R^n$, $h$ is the standard Euclidean metric, and we can replace $\star_M$ by the standard convolution $\ast$ on $\R^n$ (cf.\ the notational conventions in Section \ref{ssec:mfreg}). 
Since \eqref{eq:ric_pos_assumption} is supposed to hold for $g$-timelike vector fields, we may without loss of generality assume that $C>0$. 
We will follow the basic layout of the proof of \cite[Lem.\ 3.2]{KSSV:15}.

To begin with, by uniform continuity of $g$ on $K$ there exists some $r>0$ such that, for all $p,x\in K$ with $\|p-x\|<r$ and
any $X\in \R^n$ with $\|X\|\le D$, we have $\|g_p(X,X)-g_x(X,X)\| < -\eta$. Fixing $p\in K$, it follows that on the open ball $B_r(p)$ the constant vector field $\tilde X:= x\mapsto X$ is $g$-timelike. Choose a cut-off function $\chi\in C^\infty_c(\R^n)$ with $\chi\equiv 1$ in a neighbourhood of $\overline{B_r(p)}$ and set, for $1\le i,j\le n$, $\tilde R_{ij} := \chi \cdot\Ric[g]_{ij} \in \D'(\R^n)$.  Due to \eqref{eq:ric_pos_assumption}, for any $x\in B_{r-1/k}(p)$ we have 
\begin{align}\label{eq:tilde_R}
(\tilde R_{ij}\tilde X^i\tilde X^j)*\rho_k(x) \ge (Cg_{ij}\tilde X^i \tilde X^j)*\rho_k(x) = C (g_k)_{ij}(x)X^iX^j.
\end{align}
Now note that for $1/k < r$ we have $(\Ric[g]_{ij}*\rho_k)(p) = (\tilde R_{ij}*\rho_k)(p)$. Thus for $k>1/r$ we obtain,
using the constancy of $\tilde X$,
\begin{equation}\label{eq:loc_Ric_est_uniform}
\begin{split}
|\Ric[g_k]_{ij}(p)X^iX^k -  ((\tilde R_{ij}\tilde X^i \tilde X^j)*\rho_k)(p)| &= |(\Ric[g_k]_{ij}(p) - \Ric[g]_{ij}*\rho_k(p))X^iX^j| \\
&\le D^2\; \max_{i,j} \,\sup_{x\in K}|\Ric[g_k]_{ij}(x) - \Ric[g]_{ij}*\rho_k(x)|.
\end{split}
\end{equation}
Using \eqref{eq:tilde_R} we arrive at
\begin{equation}\label{eq:pre-final}
\begin{split}
\Ric[g_k]_{ij}(p)X^iX^j &= (\tilde R_{ij}X^iX^j)* \rho_k (p) + (\Ric[g_k]_{ij}(p)X^iX^j - (\tilde R_{ij}X^iX^j)* \rho_k (p)) \\
            &\ge C (g_k)_{ij}(p)X^iX^j - |\Ric[g_k]_{ij}(p)X^iX^j -  ((\tilde R_{ij}\tilde X^i \tilde X^j)*\rho_k)(p)|\\
            &=: \Box_1 + \Box_2.
\end{split}    
\end{equation}
Employing \eqref{eq:loc_Ric_est_uniform}, together with Proposition \ref{Pr:Ricci Lp-convergence}(ii), we see that
\[
\Box_2 \ge -C_K D^2 \ge - \frac{C_K D^2}{\eta} g_{ij}(p)X^iX^j \ge  -\frac{2C_K D^2}{\eta} (g_k)_{ij}(p)X^iX^j
\]
for $k$ large since $g_k(X,X) \to g(X,X)$ uniformly on $K\times \overline{B_D(0)}$. Combining this
with \eqref{eq:pre-final} gives the claim with $\tilde C := C - \frac{2C_K D^2}{\eta}$.
\end{proof}
In the following result, we call a sequence $(X_k)$ of vector fields on $M$ \emph{locally uniformly timelike} if for each $K\Subset M$ there exists some $c<0$ such that, for each $k\in \N$, $g(X_k,X_k)<c$ on $K$.

\begin{lem}[$L^p$-convergence for ${\Ric[g_k]}$]\label{Lem:(iv) of our Hawking theorem} 
Let $\Ric[g](X,X)\geq \rho g(X,X)$ in distributions for some $\rho\in \R$ and each $g$-timelike $X\in \mathfrak{X}(M)$ and suppose that $(X_k)$ is a sequence of smooth
locally uniformly bounded and locally uniformly $g$-timelike vector fields.
Then for all $K\Subset M$ and any $p\in [1,\infty)$ we have (with $(\,)_-$ denoting the negative part of a function)
\begin{equation}\label{eq:neg_ricci_claim}
    \big(\Ric[g_k](X_k,X_k)-\rho g_k(X_k,X_k)\big)_-\to 0\quad\mbox{in }L^p(K) 
\end{equation}
as $k\to \infty$.
\end{lem}
\begin{proof} 
As in the proof of \Cref{Pr:(i) for our Hawking theorem} we see 
that it suffices to show \eqref{eq:neg_ricci_claim} when replacing $g_k$ by $g_{\eps_k}$, which we will tacitly do below in order to refer to results in Section \ref{sec:regularisation}. 

Let $V$ be a relatively compact open neighbourhood of $K$. Without loss of generality
we may assume that on an open neighbourhood $U$ of $\Bar{V}$ there exists a smooth $h$-orthonormal frame $F_1,\dots,F_n \in \mathfrak{X}(U)$ ($h$ a smooth Riemannian background metric). Fix any $\Bar{x}\in K$ and write
\[
X_k(\Bar{x}) = \sum_{j=1}^n \alpha_j^k F_j(\Bar{x})
\]
$(\alpha_1^k,\dots, \alpha_n^k \in \R)$. Now define $\Bar{X}_k\in \mathfrak{X}(U)$ by 
\begin{equation}\label{eq:barXdef}
\Bar{X}_k = \sum_{j=1}^n \alpha_j^k F_j \qquad (k=1,\dots,n).
\end{equation}
Since the $X_k$ are uniformly bounded on $\Bar{V}$, there exists some $C_K$ independent of
$\Bar{x}$ such that 
\begin{equation}\label{eq:indep_bound}
|\alpha_j^k| \le C_K \qquad (k\in \N,\ l=1,\dots, n).
\end{equation}
Let $c<0$ be such that $g(X_k,X_k)<c$ on $\Bar{V}$ for each $k\in \N$. Using the local Lipschitz property of $g$ (as well as that of the $F_j$) together with \eqref{eq:indep_bound}, we may without loss of generality  assume that the diameter of $\Bar{V}$ is so small
\footnote{Clearly it suffices to establish the result for any compact subset of $K$ of small diameter.} 
that 
\[
|g_{\Bar{x}}(\Bar{X}_k(\Bar{x}),\Bar{X}_k(\Bar{x}))) - g_x(\Bar{X}_k(x),\Bar{X}_k(x))| < \frac{|c|}{2},
\]
hence $g_x(\Bar{X}_k(x),\Bar{X}_k(x)) < c/2$ for any $x\in \Bar{V}$ and each $k$. Thus each $\Bar{X}_k$
is $g$-timelike on all of $V$. Moreover, this property holds irrespective of the original choice of $\Bar{x}$.

Consider now the following decomposition:
\begin{equation}\label{eq:the_decomposition}
\begin{split}
\Ric[g_k](\Bar{X}_k,\Bar{X}_k) - \rho g_k(\Bar{X}_k,&\Bar{X}_k) = [\Ric[g_k] - (\Ric[g]\star_M \rho_k)](\Bar{X}_k,\Bar{X}_k) \\
&+[(\Ric[g]\star_M \rho_k)(\Bar{X}_k,\Bar{X}_k) - \big(\Ric[g](\Bar{X}_k,\Bar{X}_k)\big)\star_M\rho_k]\\
&+[\Ric[g](\Bar{X}_k,\Bar{X}_k) - \rho g(\Bar{X}_k,\Bar{X}_k)]\star_M\rho_k \\
&+ [\rho g(\Bar{X}_k,\Bar{X}_k)\star_M\rho_k - \rho g_k(\Bar{X}_k,\Bar{X}_k)]\\ 
&=: 
A_k(\Bar{X}_k,\Bar{X}_k) + B_k(\Bar{X}_k,\Bar{X}_k) + C_k(\Bar{X}_k,\Bar{X}_k) +D_k(\Bar{X}_k,\Bar{X}_k).
\end{split}
\end{equation}
Since $\Bar{X}_k$ is $g$-timelike on $V$ and
$\rho \ge 0$, we have $C_k(\Bar{X}_k,\Bar{X}_k)(\Bar{x}) \ge 0$ for all $k$ large (depending only on the distance from $K$ to $\partial V$). 
Keeping in mind that
\[
(\Ric[g_k](\Bar{X}_k,\Bar{X}_k) - \rho g_k(\Bar{X}_k,\Bar{X}_k))(\Bar{x})
= (\Ric[g_k](X_k,X_k) - \rho g_k(X_k,X_k))(\Bar{x}),
\]
we therefore obtain
\begin{equation}\label{eq:Ricci-split}
\begin{split}
[\Ric[g_k](X_k,X_k) - \rho g_k(X_k,&X_k)]_-(\Bar{x})\\
& \le
|A_k(\Bar{X}_k,\Bar{X}_k)(\Bar{x})| + |B_k(\Bar{X}_k,\Bar{X}_k)(\Bar{x})| + |D_k(\Bar{X}_k,\Bar{X}_k)(\Bar{x})|.
\end{split}    
\end{equation}
Set $\hat A_k := \max\{|A_k(F_j,F_l)| \colon j,l = 1,\dots,n \}$, and analogously define $\hat B_k$ and
$\hat D_k$. Since the coefficients in \eqref{eq:barXdef} are constant, from \eqref{eq:indep_bound} and \eqref{eq:Ricci-split} we conclude that
there exists a constant $\tilde C_K>0$ independent of $\Bar{x}$ such that
\[
[\Ric[g_k](X_k,X_k) - \rho g_k(X_k,X_k)]_-(\Bar{x}) \le \tilde C_K (\hat A_k + \hat B_k + \hat D_k)(\Bar{x}).
\]
This pointwise estimate on $K$ reduces our task to showing $L^p(K)$-convergence to $0$ of $\hat A_k$, $\hat B_k$,
and $\hat{D_k}$. Indeed, for $\hat A_k$ this follows from \Cref{Pr:Ricci Lp-convergence}(i) and for $\hat B_k$ 
it holds due to \Cref{Le:Box2 of L4.8}. Finally,
$\hat D_k\to 0$ uniformly on $K$.  
\end{proof}

We are now ready to prove our first main result, \Cref{Th: Hawking gh case}.

\begin{proof}[Proof of \Cref{Th: Hawking gh case}]

 Let 
\begin{equation}\label{eq:alpha-cases}
     \alpha\equiv \alpha(\beta,\rho):=  \begin{cases}
        \frac{n-1}{\abs{\beta}} &\text{if} \;\rho=0\\
         \frac{1}{\sqrt{|\rho|}}\coth^{-1}\Big(\frac{|\beta|}{(n-1)\sqrt{|\rho|}}\Big) &\text{if} \;\rho<0 \\
         \frac{1}{\sqrt{\rho}}\cot^{-1}\Big(\frac{|\beta|}{(n-1)\sqrt{\rho}}\Big) &\text{if} \;\rho>0
         \end{cases}
\end{equation}
 and assume $|\beta|>(n-1)\sqrt{|\rho|}$ if $\rho\ne 0$. We have to show that $\tau_\Sigma\leq \alpha$.

To this end we suppose by contradiction that there is $\hat q\in J^+(\Sigma)$ with $\tau_\Sigma(\hat q)>\alpha+3\delta$ for some $\delta>0$. Choose some $q_0\in I^-(\hat q)$ with $\alpha+\delta<\tau_\Sigma(q_0)<\alpha+2\delta$. Then there is an open and relatively compact neighbourhood $U$ of $q_0$ in $J^-(\hat q)$ such that for all $q\in \overline{U}$
\begin{equation}
    \alpha+\delta<\tau_\Sigma(q)<\alpha+2\delta.
\end{equation}
Now we start using the regularisations $g_k$ of $g$. By Lemma \ref{Le:uniform convergence of tau} $\tau_{\Sigma,k}\rightarrow\tau_\Sigma$ uniformly on $\overline{U}$ and so there is $k_0$ such that for all $k\geq k_0$ and for all $q\in U$
\begin{equation}
    \alpha+\delta<\tau_{\Sigma,k}(q)<\alpha+2\delta.
\end{equation}
Furthermore, since $\Sigma$ is a Cauchy surface for the smooth metric $g_k$,
there are base points $p_k^q\in \Sigma\cap J^-(\hat q)$ and $g_k$-maximising geodesics $\gamma^q_k$ from $p^q_k$ to $q$. 
Define 
\begin{equation}
    B_k:=\{p^q_k|\ q\in U\}\subseteq J^-(\hat q)\cap\Sigma=:N
\end{equation}
and note that $N$ is compact by Lemma \ref{Le:properties of sigma}(i). 

Consider now the initial parts of the
$g_k$-unit speed geodesics that start $g_k$-orthogonally from $N$
for as long as they stay within $L:=J^-(\hat q)\cap J^+(\Sigma)$, which is compact also by Lemma \ref{Le:properties of sigma}(i). Since these curves are $g$-timelike and $L$ is compact, their $h$-lengths are uniformly bounded, cf., e.g., \cite[Lem.\ 2.1, and proof of Lem.\ 2.2]{GLS:18}. Moreover, the $C^{0,1}$-norms of the $g_k$
are uniformly bounded on $L$. It therefore follows from \cite[Prop.\ 1.4]{LLS:21} that the $C^{1,1}$-norms of these geodesics are uniformly bounded on $L$. From this and the fact that $g_k\to g$ uniformly on $L$ it follows that the $^kU$, defined as in \eqref{eq:Up} but for $g_k$, are uniformly bounded and $g$-uniformly timelike on $L$.

By definition we also have that (for the notation see Section \ref{sec:vc})
\begin{equation}\label{eq:Bk-inclusion}
B_k\subseteq\ {^k}\mathrm{Reg}^+_\delta(\alpha)
\end{equation}and we also clearly have that $\vol_{\Sigma,k}(B_k)$ is finite. Next we show that it is also positive. 
More precisely, we will show that there is a constant $D>0$ such that for $k$ large we have
\begin{equation}\label{eq:vol-pos}
    \vol_{\Sigma,k}(B_k)\geq D. 
\end{equation}
To this end, first note that for the set  ${}^k\Omega^+_{\alpha+2\delta}(B_k)$ (cf.\ \eqref{eq:Omega_T_+}) we have for
large $k$
    \begin{equation}
        \begin{split}\label{Eq:upper area}
            \mathrm{vol}_k({}^k\Omega^+_{\alpha+2\delta}(B_k)\cap L)\geq\mathrm{vol}_k(U)>\frac{1}{2}\mathrm{vol}(U)>0
        \end{split}
    \end{equation}
(where we took into account that the cut locus $\mathrm{Cut}^k_\Sigma$ is a set of measure zero).

Next we want to estimate $ \vol_{\Sigma,k}(B_k)$ from below in terms of $\vol_k(\Omega^+_{\alpha+2\delta}(B_k)\cap L)$ to establish \eqref{eq:vol-pos}. To this end we use assumptions (i) and (ii) to give us appropriate curvature bounds needed to apply \cite[Lem.\ 3.2 and Rem.\ 3.3]{GKOS:22}. First, the distributional strong energy condition (i) together with Proposition  \ref{Pr:(i) for our Hawking theorem} and the fact that the $^kU$ are uniformly bounded and uniformly $g$-timelike 
implies the existence of some $k_0$ and some $\kappa<0$ such that, for all $k\ge k_0$ we have 
\begin{equation}\label{eq:rbb}
     \Ric[g_k](^kU,{}^kU)\geq  \kappa
\end{equation}
on $L$.
Furthermore, assumption (ii) in conjunction with \Cref{Le: ConvMeanCurv} implies that, for large $k$,
\begin{align}\label{Eq:FINAL2meancurv}
H[g_k]{|}_N<\beta<0, 
\end{align}
where $H[g_k]$ is the mean curvature of $\Sigma$ with respect to the smooth metric $g_k$. 

Now, writing the volume measure of $g_k$ in a normal exponential chart 
as $\vol_k=\mathcal{A}_k\,\dint t\otimes \dint\sigma_k$ with $\sigma_k$ the Riemannian measure induced by $g_k$ on $\Sigma$,  we have upon choosing coordinates on $\Sigma$ that $\mathcal{A}_k(t,x)={\sqrt{\det g_k(t,x)}}/{\sqrt{\det g_k(0,x)}}$, cf.\ \cite[Rem. 2.8]{GKOS:22}. By increasing the modulus of $\kappa$ in \eqref{eq:rbb} if necessary we may assume that $\beta\geq-(n-1)\sqrt{|\kappa|}$. 
Finally, denoting by $\gamma_x^k$ the $g_k$-unit speed geodesic starting $g_k$-orthogonally from $x\in N$, let ${}^kf_L(x) := \sup\{s\in [0,\infty) \colon \gamma^k_x([0,s])\subseteq L\}$.

Then by \cite[Lem.\ 3.2 and Rem.\ 3.3]{GKOS:22} and noting that the proofs of these results rely exclusively on the bounds \eqref{eq:rbb}, \eqref{Eq:FINAL2meancurv}, 
we obtain 
\begin{equation}
        \begin{split}\label{Eq:lower area}
            \mathrm{vol}_k({}^k\Omega^+_{\alpha+2\delta}(B_k)\cap L)& 
            =\int_{B_k}\int_{0}^{\min(\alpha+2\delta,{}^k c^+_\Sigma(x),{}^kf_L(x))}\mathcal{A}_{k}(t,x) 
            \,\dint t\,\dint\sigma_k \\
            & \le C \int_{0}^{\alpha+2\delta}\int_{B_k}\, \dint\sigma_k\,\dint t
            =C\cdot (\alpha+2\delta)\cdot\mathrm{vol}_{\Sigma,k}(B_k),
            \end{split}
    \end{equation}
    with $C=C(n,\kappa,\beta,\alpha+2\delta)$.
    We then conclude \eqref{eq:vol-pos} by combining \eqref{Eq:upper area} and \eqref{Eq:lower area}.
\medskip
  
At this point we want to apply Theorem \ref{thm:sat} to $(M,g_k)$ for $B_k$ and the set  ${^k}\Omega^+_{\alpha+\delta/2}(B_k)$ for large $k$. 
To this end, we first note that, again, inspection of the proof of that result in \cite{GKOS:22} shows that the Ricci- resp.\ mean curvature bounds assumed globally in its formulation  can  in fact be replaced by \eqref{eq:rbb}, \eqref{Eq:FINAL2meancurv} in our setting (as long as $\smash{{^k}\Omega^+_{T+\eta}(B_k)\subseteq L}$ in the notation of Theorem \ref{thm:sat}).

Next, note that ${^k}\Omega^+_{\alpha+\delta}(B_k)\subseteq L$, where \eqref{eq:rbb} and \eqref{Eq:FINAL2meancurv} hold with constants $\kappa$ and $\beta$  satisfying $\smash{\beta\geq-(n-1)\sqrt{|\kappa|}}$. Next we choose $\eta:=\delta/2$, and $T:=\alpha+\eta$. 
Since $K(\beta,T,\rho)$ from Theorem \ref{thm:sat} is strictly increasing in $T$ we have that $K:=K(\beta,T,\rho)>K(\beta,\alpha(\beta,\rho),\rho)=0$. In case $\rho>0$ we may assume $\delta$ is small enough such that $\sqrt{\rho}T<\frac{\pi}{2}$.

Therefore (i), together with 
Lemma \ref{Lem:(iv) of our Hawking theorem} implies that the $L^1$-norm of the negative part of $\Ric[g_k]$ on the geodesic tangents $^kU$ goes to zero, so for large $k$ we obtain
\begin{align}
    \frac{1}{\mathrm{vol}_{\Sigma,k}(B_k)} \int_{{^k}\Omega^+_{\alpha+\eta}(B_k)} \Big(\Ric[g_k](^k U_p,\,^k U_p)-(n-1)\rho \Big)_{-}\,\dint\mathrm{vol}_k(p)
    < C^{A-}(n,\kappa,\eta,T)\, K.
\end{align}
But now Theorem \ref{thm:sat} gives $B_k\not\subseteq{^k}\mathrm{Reg}^+_{\delta/2}(\alpha+\delta/2)={^k}\mathrm{Reg}^+_{\delta}(\alpha)$, which contradicts \eqref{eq:Bk-inclusion}.
\end{proof}

    \begin{rem}[Comparison with synthetic versions of the Hawking singularity theorem]     
    In \cite[Thm.\ 5.6]{CM:20}, a synthetic Hawking singularity theorem was proved for globally hyperbolic timelike non-branching Lorentzian geodesic spaces satisfying the strong energy condition in the sense of the timelike measure contraction property $\tmcp^e(K,N)$. To compare this result, first, to the $C^1$-version of \cite[Thm.\ 4.13]{Gra:20}, we note that although \Cref{lemma:eq_meancurvature} remains true also for $C^1$-metrics, the question about what the right-hand side of \eqref{Eq: halpha} represents still remains. In addition, we need to compare the ways in which the strong energy condition is modelled in each case. In this regularity, it was shown in \cite{BC:22} that distributional Ricci bounds for $C^1$-metrics indeed imply
    the bounds used in \cite[Thm.\ 5.6]{CM:20}.

    In the Lipschitz case treated in the current work, apart from what was said on mean curvature bounds in \Cref{Rem: mean curv comp}, even the compatibility of the notions of timelike Ricci curvature bounds is still open. 
    The situation is better in the Riemannian context, where the problem was first studied for the $C^1$-case in \cite{KOV:22}, and more recently compatibility was shown in \cite{MR:24} even down to metrics of regularity $C^0\cap W^{1,2}_{\mathrm{loc}}$.  
    
    Another very recent synthetic Hawking singularity theorem was shown in \cite[Theorem A.7]{BMcC:23}. Here, the authors use the same version of mean curvature bounds as in \cite{CM:20}, and also implement the strong energy condition via the timelike measure contraction property $\tmcp^e(K,N)$. Thus the same compatibility statements as 
    detailed above for \cite{CM:20} apply here as well.
    The Hawking theorem itself is then based on integral curvature bounds. Even in the weighted smooth case, this is a new result.

    Note that, in any case, since both \cite[Thm.\ 5.6]{CM:20} and \cite[Thm.\ A.7]{BMcC:23} assume an (essential) non-branching condition, even if the curvature bounds turn out to be equivalent in all cases, the synthetic versions of the theorem will not imply our results since the latter do not rely on any non-branching condition.
   \end{rem}
\subsection{The non-globally hyperbolic case}

In this final section we are going to state and prove the extension of \cite[Thm.\ 14.55B]{ONe:83} to Lipschitz spacetimes. 
Beforehand we establish that maximisers emanating from a spacelike hypersurface $\Sigma$ start orthogonal to it, cf.\ \cite[Rem.\ 6.6(ii)]{GGKS:18} for the null case and $g\in C^{1,1}$. 

\begin{lem}[Maximisers  start orthogonally from $\Sigma$]\label{Lem: max to sigma}
  Let $(M,g)$ be a spacetime with locally Lipschitz metric tensor $g$ and let $\Sigma$ be a smooth spacelike hypersurface. Then any maximiser emanating from $\Sigma$ starts orthogonally to $\Sigma$.
\end{lem}

\begin{proof} By Remark \ref{rem:geomax}(ii), any maximiser has a parametrization as a $C^{1,1}$-curve.
    Let $\gamma_v:[0,1]\rightarrow M$ be such a curve with $\gamma_v(0)=p\in\Sigma$ and $\dot\gamma_v(0)=v$. We suppose by contradiction that $v\not\in T_p\Sigma^\perp$ and construct a variation of $\gamma_v$ which gives a longer curve from $\Sigma$ to $q  =\gamma_v(1)$.
    
    Since $\gamma_v\not\perp\Sigma$ 
    there is $y\in T_p\Sigma$ with $g(y,v)>0$. Let $\alpha:[0,b]\rightarrow\Sigma$ be a $C^2$-curve with $\alpha(0)=p$ and $\dot\alpha(0)=y$. Since this is a local issue we may assume $M=\R^n$ and $\alpha(0)=p=\gamma_v(0)=0$. 
    Now we define a variation of $\gamma_v$ by $\sigma:[0,t_0]\times[0,s_0]\rightarrow\R^n$, 
    \begin{align}
        \sigma(t,s)=\gamma_v(t)+\left(1-\frac{t}{t_0}\right)\alpha(s),
    \end{align}
    where we have chosen $t_0$ and $s_0$ so small that
    \begin{align}\label{Eq:(star)}
    \smash{\left\langle y,\dot\gamma_v(t)\right\rangle_{\sigma(s,t)}>c>0}\qquad\forall(t,s)\in[0,t_0]\times[0,s_0],
    \end{align}
    
    To show that $\sigma(\cdot,s):[0,t_0]\rightarrow\R^n$ is a longer curve from $\alpha(s)\in\Sigma$ to $\sigma(t_0,s)=\gamma_v(t_0)$ for $s,t_0$ small enough, we first Taylor expand $\alpha$ to obtain $\alpha(s)=sy+O(s^2)$. Then we compute 
    \begin{equation}
        \begin{aligned}\label{Eq:(starstar)}
            \vert g(\sigma(t,s))-\smash{g(\underbrace{\gamma_v(t)}_{\sigma(t,0)})}\vert&\leq\Lip(g)\ \vert\sigma(t,s)-\sigma(t,0)\vert\\
            &\leq\Lip(g)\left(1-\frac{t}{t_0}\right)\,|\alpha(s)|\leq C\cdot s\left(1-\frac{t}{t_0}\right),
        \end{aligned}
    \end{equation}
 for $s$ small enough. Moreover, we have $\smash{\partial_t\sigma(t,s)=\dot\gamma_v(t)-\tfrac{1}{t_0}\alpha(s)=\dot\gamma_v(t)-\tfrac{s}{t_0}y+O(s^2)}$ and so
    \begin{equation}
        \begin{aligned}
            \left\langle\partial_t\sigma(t,s),\partial_t\sigma(t,s)\right\rangle_{\sigma(t,s)}&=\left\langle\dot\gamma_v(t),\dot\gamma_v(t)\right\rangle_{\sigma(t,s)}-2\ \frac{s}{t_0}\left\langle\dot\gamma_v(t),y\right\rangle_{\sigma(t,s)}+O(s^2)\\
            &\leq\left\langle\dot\gamma_v(t),\dot\gamma_v(t)\right\rangle_{\gamma_v(t)}\ +\  
            s\left(C\Big(1-\frac{t}{t_0}\Big)-2\ \frac{c}{t_0}\right) +O(s^2).
        \end{aligned}
    \end{equation}
    Now for $s,t_0$ small enough the sum of the trailing terms on the right hand side becomes negative and so
    $L(\sigma(\cdot,s))>L(\gamma_v\vert_{[0,t_0]})$. 
  \end{proof}

Our second main result can now be stated and proven:
\begin{thm}[$C^{0,1}$-Hawking singularity theorem, II]\label{Thm: Hawking non globally hyperbolic}
    Let $(M,g)$ be a spacetime with locally Lipschitz metric tensor $g$ such that:
    \begin{enumerate}
        \item There exists  
        $\rho\in \R$
        such that $\Ric_g(X,X)\geq -(n-1)\rho\, g(X,X)$ in the distributional sense for all timelike $X\in\mathfrak{X}(M)$.
        \item There is a smooth compact spacelike hypersurface $\Sigma$ with $\mathcal{H}[g]< \beta<0$, where in case $\rho<0$ we assume  $|\beta|>(n-1)\sqrt{|\rho|}$.
    \end{enumerate}
  Then there exists a timelike future directed geodesic $\gamma$ in the sense of Filippov emanating orthogonally from $\Sigma$ which is incomplete. More precisely, the length of this $\gamma$ is bounded above by the right-hand side of \eqref{eq:Hawking-bounds-cases}.
\end{thm}

\begin{proof}
    Similar to the classical proof (cf.\ \cite[Thm.\ 14.55B]{ONe:83}) we employ Theorem \ref{Th: Hawking gh case} to establish the result. More precisely, our arguments rely on Theorem \ref{Th: Hawking gh case} only via the bound on $\tau_\Sigma$ on the interior of $D_g(\Sigma)$. It will therefore suffice to treat the case $\rho=0$, with the understanding that the other cases 
    follow entirely analogously, simply replacing $\smash{\tfrac{n-1}{\abs{\beta}}}$ by the appropriate bound on the right hand side of \eqref{eq:Hawking-bounds-cases}.
    
    Thus, supposing that $\rho=0$, we are going to show that there exists an inextendible future directed timelike Filippov geodesic $\gamma$ starting orthogonally from $\Sigma$ with length bounded above by $\smash{\tfrac{n-1}{\abs{\beta}}}$.

  To begin with, the covering argument given in \cite[Prop.\ 14.48]{ONe:83} allows us also in the present case to assume, without loss of generality, that $\Sigma$ is acausal and connected. By \cite[Thm.\ 5.7]{Sae:16} the interior of the Cauchy development $\mathrm{int}(D_g(\Sigma))$ is globally hyperbolic and we may apply \Cref{Th: Hawking gh case} to see that $\smash{\tau_\Sigma\leq\tfrac{n-1}{\abs{\beta}}}$ on this set.
 
  If the future Cauchy horizon $H^+_g(\Sigma)=\emptyset$ then by \cite[Thm.\ 2.35]{Min:19a} we obtain that $I^+(\Sigma) \subseteq D^+(\Sigma)$, so the
  conclusion follows from \Cref{Th: Hawking gh case}.  Thus from now on we assume that 
  \begin{align}
  H^+_g(\Sigma)\neq\emptyset.
  \end{align}
 
We make the \textbf{indirect assumption} that 
every inextendible timelike future directed  Filippov geodesic starting perpendicularly from $\Sigma$ 
has length strictly greater than $\smash{\tfrac{n-1}{\abs{\beta}}}$.
  
  \textbf{Step 1.} \emph{Compactness of the horizon.} The set $B$ of initial conditions for these geodesics is given by $\{\vec{n}_p\mid p\in\Sigma\}$ with $\vec{n}$ the future directed unit normal to $\Sigma$, hence is compact. Since by assumption all such $\gamma$ exist at least on $\smash{[0,\tfrac{n-1}{\abs{\beta}}]}$, by \cite[Thm.\ 3, p.\ 79]{Fil:88} there exists $L\Subset M$ compact containing all their images.
  Thanks to the Avez-Seifert theorem \cite[Prop.\ 6.4]{Sae:16} any point $q\in \mathrm{int}(D^+_g(\Sigma))$ is reached by a maximiser $\gamma$, which by Remark \ref{rem:geomax}(ii) and Lemma \ref{Lem: max to sigma} is a Filippov geodesic emanating orthogonally from $\Sigma$. Moreover, by Theorem \ref{Th: Hawking gh case} its length is bounded by $\smash{\tau_\Sigma\leq\tfrac{n-1}{\abs{\beta}}}$,
  so we conclude that int$(D^+_g(\Sigma))\subseteq L$. 
  
  Due to the fact that $D^+_g(\Sigma)\setminus\Sigma$ is open by \cite[Thm.\ 2.34]{Min:19a},
  we get $D^+_g(\Sigma)\setminus\Sigma\subseteq\mathrm{int}(D^+_g(\Sigma))$, and consequently $\smash{\overline{D^+_g(\Sigma)}\subseteq L}$.
  Hence  (again using \cite[Thm.\ 2.35]{Min:19a}) $H^+_g(\Sigma)\subseteq L$ and so it is compact as well.

 \textbf{Step 2.} \textit{Any point in $H^+_g(\Sigma)$ can be reached by a maximiser.}
Retaining the notation from \eqref{eq:tauk} consider an approximating sequence of smooth metrics $g_k$ with $g_k\prec g$. Each $g_k$ is globally hyperbolic on $\mathrm{int}(D_{g_k}(\Sigma)) = D_{g_k}(\Sigma)$ (cf.\ (the proof of) \cite[Lem.\ 14.43]{ONe:83}).

(1) We claim that for all $k$
\begin{equation}\label{eq:Dplusorder}
H_g^+(\Sigma) \subseteq  \overline{D^+_g(\Sigma)} \subseteq
D_{g_{k+1}}^+(\Sigma) \subseteq D_{g_k}^+(\Sigma).
\end{equation} 
Indeed, by \cite[Thm.\ 2.36]{Min:19a} we have
\begin{align}
    \overline{D^+_g(\Sigma)}=\tilde{D}^+_g(\Sigma):=\{q\mid\mbox{ every past inextendible timelike curve from }q\mbox{ intersects }\Sigma\}.
\end{align}
Let $q\in\tilde{D}^+_g(\Sigma)$ and let $\gamma$ be $g_k$-past directed, past inextendible and causal. Then $\gamma$ is also $g$-past directed timelike and hence must intersect $\Sigma$, so $\smash{\overline{D^+_g(\Sigma)}\subseteq D^+_{g_k}(\Sigma)}$ and \eqref{eq:Dplusorder} follows. 

(2) Due to the classical Avez-Seifert theorem, given $q\in H^+_g(\Sigma)$, for each $k\in\N$ there exists a $g_k$-maximiser $\gamma_k$  
from $\Sigma$ to $q$ with
  \begin{align}
      L_{g_k}(\gamma_k)=\tau_{\Sigma,k}(q).
  \end{align}
  We assume $\gamma_k$ to be parametrised with respect to $h$-arclength.
  Because $\gamma_k$ is $g_k$-causal it is $g$-timelike and so must be contained in $\smash{\overline{D_g^+(\Sigma)}}$ since it terminates in $H_g^+(\Sigma)$.
Denoting by $p_k$ the initial point of $\gamma_k$, we may without loss of generality suppose that $p_k\to p\in \Sigma$. 
Let $\hat{g}_k:=\hat{g}_{\eps_k}$ be a sequence of smooth Lorentzian metrics as in \Cref{Le:approximating metrics} approximating $g$ from the outside. 

(3) Fixing $m\in \N$ we now want to apply the limit curve theorem \cite[Thm. 3.1(2)]{Min:08a} to the sequence $\{\gamma_k\}_k$ of $\hat{g}_m$-timelike curves in $(M,\hat{g}_m)$ to obtain a limit curve that connects $p$ and $q$. To do so 
we have to exclude case (2)(ii) in that theorem. Note that we cannot immediately conclude this from the curves being contained in a compact set since $(M,\hat{g}_m)$ need not be non-totally imprisoning. 

Assume, to the contrary, that the $h$-arclengths of the curves $\gamma_k$ are unbounded, i.e., $\gamma_k:[0,b_k] \to M$ with $\gamma_k(b_k)=q$ and $b_k\to \infty$. Then we obtain a future-inextendible $\hat{g}_m$-causal limit curve $\gamma:[0,\infty)\to M$ emanating from $p$. Note that this limit curve is in fact independent of $m$ and is $\hat{g}_m$-causal for all $m$, hence $g$-causal.

(4) We claim that for any $t$ with $\gamma(t)\in D_g^+(\Sigma)$, $\gamma$ is $g$-maximising from $\Sigma$ to $\gamma(t)$. By \Cref{Le:uniform convergence of tau} (which applies since we are in a globally hyperbolic region) and \Cref{Le:thinner lightcones}(i) we have
\begin{align*}
\tau_\Sigma(\gamma(t))=\lim_{k\to \infty}\tau_{\Sigma,k}(\gamma_k(t))=\lim_{k\to \infty} L_{g_k}(\gamma_k|_{[0,t]}) 
\leq \limsup_{k\to\infty} L_g(\gamma_k|_{[0,t]}). 
\end{align*}
Since the $\hat{g}_m$ approximate $g$ from the outside, for each $k$ and $m$ we 
have (using \cite[Lem.\ A.1]{McCS:22}) $L_{g}(\gamma_k|_{[0,t]})\le L_{\hat{g}_m}(\gamma_k|_{[0,t]})$. Together with \cite[Thm.\ 2.4(b)]{Min:08a} this implies
\[
\limsup_k L_g(\gamma_k|_{[0,t]}) \le \limsup_k L_{\hat{g}_m}(\gamma_k|_{[0,t]}) \le L_{\hat{g}_m}(\gamma|_{[0,t]}).
\]
This holds for each $m$, and since $L_{\hat{g}_m}(\gamma|_{[0,t]}) \to L_g(\gamma|_{[0,t]})$ as $m\to\infty$, 
altogether we arrive at
\[
\tau_\Sigma(\gamma(t)) \le L_g(\gamma|_{[0,t]}),
\]
proving the claim.

(5) Next we show that $T:=\sup\{t\in [0,\infty) \colon \gamma(t)\in D^+_g(\Sigma)\}<\infty$. Indeed, if this were not the case then $\gamma\subseteq D^+_g(\Sigma)$ and by the above $\gamma$ is always maximising the Lorentzian distance to $\Sigma$. Hence by \Cref{rem:geomax}(ii) and \Cref{Lem: max to sigma} its re-parametrisation to $g$-unit speed is a Filippov-geodesic starting orthogonally to $\Sigma$. Since $\gamma$ is future inextendible and we assume Filippov-geodesic completeness, also its $g$-unit speed parametrisation must be defined on $[0,\infty)$. But then for any $\smash{t>\frac{n-1}{|\beta|}}$ we have that $\smash{\tau_\Sigma(\gamma(t))=t>\frac{n-1}{|\beta|}}$ (where $\gamma$ is being parametrized by $g$-unit speed), contradicting  $\smash{\tau_\Sigma\le \frac{n-1}{|\beta|}}$ on $D^+_g(\Sigma)$.\medskip 

(6) Since $D^+_g(\Sigma)\setminus \Sigma$ is open, we have
$\smash{\gamma(T)\in \overline{D^+_g(\Sigma)}\setminus D^+_g(\Sigma)=H^+_g(\Sigma)}$ (using \cite[Thm.\ 2.35]{Min:19a}).  
We now claim that $\gamma(t)\in I_g^+(H_g^+(\Sigma))$ for all $t>T$. 

Fix a globally hyperbolic neighbourhood $U$ of $\gamma(T)$ and fix a $\delta>0$ such that $\gamma([T-\delta,T+\delta])\subseteq U$. By $h$-uniform convergence also $\gamma_k([T-\delta,T+\delta])\subseteq U$ for large enough $k$. Thus by the same reasoning as in (4)
we conclude that $\gamma|_{[T-\delta,T+\delta]}$ is maximising, hence it has a causal character, namely timelike because it is timelike initially.

Since $I^+_g(H_g^+(\Sigma))$ is open, we also have that $\gamma_k(T+1)\in I^+_g(H_g^+(\Sigma))$ for $k$ large, contradicting the fact that $\smash{\gamma_k\subseteq \overline{D^+_g(\Sigma)}}$ (cf.\ \cite[Thm.\ 2.32]{Min:19a}). 

Altogether, we have established that case (ii) of \cite[Thm. 3.1 (2)]{Min:08a} cannot occur, hence $\gamma:[0,b]\to M$ indeed reaches $q=\gamma(b)$. By the same arguments as above, $\gamma|_{[0,b)}\subseteq D_g^+(\Sigma)$ and $\tau_\Sigma(\gamma(t))=L_g(\gamma|_{[0,t]})$ for all $t<b$. Letting $t\nearrow b$ and using lower semicontinuity of $\tau_\Sigma$ we conclude that
$\gamma$ is maximising from $\Sigma$ to $q$, thereby concluding Step 2.

\textbf{Step 3.} \textit{The function $p\mapsto\tau_\Sigma(p)$ is strictly decreasing on past-pointing generators of $H^+_g(\Sigma)$.} 
We
first note that by \cite[Thm. 2.32]{Min:19a} every $q\in H^+_g(\Sigma)$ is the future endpoint of a null maximiser $\alpha$, so we fix such a curve and take $s<t$ in its domain of definition. By the above there is a past-pointing timelike geodesic $\sigma$ from  $\alpha(t)$ to $\Sigma$ such that $L_g(\sigma)=\tau_\Sigma(\alpha(t))$. But $\alpha$ is null and so the curve $\beta$ obtained by the concatenation of $\smash{\alpha\vert_{[s,t]}}$ and $\sigma$ is not maximising by \cite[Thm.\ 1.1]{GL:18}. Hence
\begin{align}
    \tau_\Sigma(\alpha(s)) > L_g(\beta) = L_g(\sigma)=\tau_\Sigma(\alpha(t)).
\end{align}

\textbf{Step 4.} \textit{Conclusion.} Since $H^+(\Sigma)$ is compact the function $p\mapsto\tau_\Sigma(p)$ attains a minimum at some $q\in H^+(\Sigma)$. But this contradicts strict monotonicity of $\tau_\Sigma$ along a generator starting in $p$, thereby concluding the proof.
\end{proof} 

\begin{rem} 
Arguing as in the proof of Theorem \ref{Thm: Hawking non globally hyperbolic} it follows that assumption (ii) in 
Theorem \ref{Th: Hawking gh case} can be weakened to $\Sigma$ merely being a smooth \emph{future} Cauchy hypersurface satisfying the mean curvature bound. Here, by $\Sigma$ being a future Cauchy surface we mean (cf.\ \cite[p.\ 432]{ONe:83}) that $H^+(\Sigma)=\emptyset$. Indeed, given this, we may apply Theorem \ref{Th: Hawking gh case} to the globally hyperbolic spacetime $\mathrm{int}(D_g(\Sigma))$
and note that since $I^+(\Sigma) \subseteq D_g^+(\Sigma)$, the right hand side of \eqref{eq:Hawking-bounds-cases}
is in fact a global bound on $\tau_\Sigma$. 
\end{rem}

\section*{Data availability statement}
There is no data associated with this manuscript.

\section*{Conflict of interest statement} The authors have no conflicts of interest to declare.

\section*{Acknowledgments}
We would like to thank Mathias Braun, Nicola Gigli, Christian Ketterer, Christian Lange and Clemens Sämann for helpful discussions. This research was funded in part by the Austrian Science Fund (FWF) [Grants DOI \href{https://doi.org/10.55776/PAT1996423}{10.55776/PAT1996423}, \href{https://doi.org/10.55776/P33594}{10.55776/P33594}, and \href{https://doi.org/10.55776/EFP6}{10.55776/EFP6}]. MG and EH acknowledge the support of
the German Research Foundation through the excellence
cluster EXC 2121 ”Quantum Universe” – 390833306.  For open access purposes, the authors have applied a CC BY public copyright license to any author accepted manuscript version arising from this submission.

\addcontentsline{toc}{section}{References}

\end{document}